\documentclass[final,1p]{elsarticle}

\usepackage[utf8]{inputenc}
\DeclareUnicodeCharacter{012D}{\u{\i}}
\DeclareUnicodeCharacter{012D}{\u\i}

\usepackage{amsmath,amssymb}
\usepackage{amstext,amsthm}
\usepackage{latexsym}
\usepackage{amsopn}
\usepackage{color}
\usepackage[notref,notcite]{showkeys}
\usepackage{xspace}
\usepackage{stmaryrd} % fuer dreifache Normstriche
\usepackage{url}
\journal{Journal of Approximation Theory}

\DeclareMathOperator{\sinc}{sinc}
\DeclareMathOperator{\sgn}{sgn}
\DeclareMathOperator{\dist}{dist}

\DeclareMathOperator{\Lip}{Lip}
\DeclareMathOperator{\lip}{lip}
\newcommand{\abs}[1]{\left|#1\right|}
\newcommand{\zit}[1]{\textup{(\ref{#1})}}
\newcommand{\rez}[1]{\frac{1}{{#1}}}
\newcommand{\bgl}[1]{\begin{equation}\label{#1}}
\newcommand{\egl}{\end{equation}}
\newcommand{\ri}{\right}
\newcommand{\li}{\left}
\newcommand{\D}{\displaystyle}

\newcommand{\hal}{{\textstyle \frac{1}{2}}}

\newcommand{\ol}[1]{\overline{#1}}

\newcommand{\N}{\mathbb{N}}
\newcommand{\Z}{\mathbb{Z}}
\newcommand{\eps}{\varepsilon}
\newcommand{\R}{\mathbb{R}}
\newcommand{\C}{\mathbb{C}}
\newcommand*\wh[1]{\widehat{#1}}

\newcommand{\un}{\ensuremath{\infty}}

\newcommand*{\Rhf}[1]{R_{\pi/h}^{\hbox{\rm\tiny{#1}}}f}
\newcommand*{\Rh}{R_{\pi/h}^{\hbox{\rm\tiny{PDF}}}}

\newcommand{\bin}[2]{\Big(\genfrac{}{}{0pt}{}{#1}{#2}\Big)}
\newcommand{\p}[1]{\ensuremath{( #1 )}}
\newcommand{\pbig}[1]{\ensuremath{\bigl( #1 \bigr)}}
\newcommand{\pBig}[1]{\ensuremath{\Bigl( #1 \Bigr)}}
\newcommand{\pbigg}[1]{\ensuremath{\biggl( #1 \biggr)}}

\newcommand{\bk}[1]{\ensuremath{[ #1 ]}}
\newcommand{\bkbig}[1]{\ensuremath{\bigl[ #1 \bigr]}}\let\bigbk\bkbig

\newcommand{\br}[1]{\ensuremath{\{ #1 \}}}
\newcommand{\brbig}[1]{\ensuremath{\bigl\{ #1 \bigr\}}}
\newcommand{\brBig}[1]{\ensuremath{\Bigl\{ #1 \Bigr\}}}
\newcommand{\brbigg}[1]{\ensuremath{\biggl\{ #1 \biggr\}}}\let\biggbr\brbigg
\newcommand{\brBigg}[1]{\ensuremath{\Biggl\{ #1 \Biggr\}}}

\newcommand{\absbig}[1]{\ensuremath{\bigl| #1 \bigr|}} \let\bigabs\absbig
 \let\Bigabs\absBig

\newcommand{\norm}[1]{\ensuremath{\lVert #1 \rVert}}
\newcommand{\bignorm}[1]{\big\|#1\big\|}
\newcommand{\Bignorm}[1]{\Big\|#1\Big\|}
\newcommand{\biggnorm}[1]{\bigg\|#1\bigg\|}
\newcommand{\Biggnorm}[1]{\Bigg\|#1\Bigg\|}
\newcommand{\Oh}{\mathcal O}
\newcommand{\oh}{o}
\newcommand{\ds}{\displaystyle}

\newcommand{\Lone}{\ensuremath{L^1(\R)}}
\newcommand{\Ltwo}{\ensuremath{L^2(\R)}}
\newcommand{\Lq}{\ensuremath{L^q(\R)}}
\newcommand{\Wrtwo}{\ensuremath{W^{r,2}(\R)}}
\newcommand{\Wstwo}{\ensuremath{W^{s,2}(\R)}}

\newcommand{\Rieszalpha}{\ensuremath{H^\alpha_2}}
\newcommand{\Rieszbeta}{\ensuremath{H^\beta_2}}
\newcommand{\Rieszgamma}{\ensuremath{H^\gamma_2}}
\newcommand{\Rieszr}{\ensuremath{H^r_2}}
\newcommand{\Rieszhal}{\ensuremath{H^{1/2}_2}}

\newcommand{\Fstwo}{\ensuremath{F^{s,2}}}
\newcommand{\dt}{\,dt}
\newcommand{\du}{\,du}
\newcommand{\dv}{\,dv}

\newcommand{\Rae}{R^{\{\alpha\}}_{2j,\eps}\kern-1pt}

\newcommand{\Mn}{M^{2,1}_\text{\rm a}}
\newcommand{\Mtwo}{M^{2,1}}
\newcommand{\Nh}{\mathcal{N}_h}
\newcommand{\CN}{\mathcal{N}}

\newcommand*{\sm}[1]{\biggbr{\rez{h}\int_{n/h}^{(n+1)/h}\bigabs{\wh{f}_{#1}(v)}^2\,dv}^{1/2}}

\newcounter{zaehler}

\def\neuerlabel{(\roman{zaehler})\hfil}
\newdimen\azlabelsep
\newdimen\aztopsep
\newdimen\azitemsep
\newenvironment{aufzaehlung}[1]%
{\begin{list}{}%
{\usecounter{zaehler}%
\settowidth{\labelwidth}{#1}
\leftmargin\labelwidth
\labelsep\azlabelsep
\addtolength{\leftmargin}{\labelsep}
\topsep\aztopsep
\itemsep\azitemsep
\renewcommand{\makelabel}[1]{\ifx##1\empty\neuerlabel\else ##1\fi}}}%
{\end{list}}

\newcommand{\ie}{i.\,e.\xspace}
\newcommand{\eg}{e.\,g.\xspace}

\theoremstyle{plain}% default
\newtheorem{theorem}{Theorem}[section]
\newtheorem{lemma}[theorem]{Lemma}
\newtheorem{proposition}[theorem]{Proposition}
\newtheorem{corollary}[theorem]{Corollary}

\theoremstyle{definition}
\newtheorem{definition}[theorem]{Definition}

\theoremstyle{remark}
\newtheorem*{remark}{Remark}

\begin{document}

\begin{frontmatter}



\title{Sobolev Spaces of Fractional Order, Lipschitz Spaces, Readapted Modulation Spaces and Their Interrelations; Applications\\[2ex]
\small\textit{Dedicated to Jacob Lionel Bakst Cooper and Hans Georg Feichtinger,\\ the grandfather and father of modulation spaces}}

\author[author1]{Paul L.~Butzer}
\ead{butzer@rwth-aachen.de}
%\address{Lehrstuhl A für Mathematik, RWTH Aachen,
%		52056 Aachen, Germany}

%

\author[author2]{Gerhard~Schmeisser}
\ead{schmeisser@mi.uni-erlangen.de}

\author[author1]{Rudolf L.~Stens\corref{cor1}}
\ead{stens@mathA.rwth-aachen.de}	
 \cortext[cor1]{Corresponding author}

\address[author1]{Lehrstuhl A für Mathematik, RWTH Aachen,
		52056 Aachen, Germany}

\address[author2]{Department of Mathematics, University of Erlangen-Nuremberg,\\
		91058 Erlangen, Germany}		

\begin{abstract}
The purpose of this investigation is to extend  basic equations and inequalities which hold for functions $f$ in a  Bernstein space $B_\sigma^2$ to larger spaces by adding a remainder term which involves the distance of $f$ from $B_\sigma^2$.

First we present a modification of the classical modulation space $M^{2,1}(\R)$, the so-called readapted modulation space $\Mn(\R)$. Our approach to the latter space and its role in functional analysis is novel. In fact, we establish several chains of inclusion relations between $\Mn$ and the more common Lipschitz and Sobolev spaces, including Sobolev spaces of fractional order.

Next we introduce an appropriate metric for describing the distance of a function belonging to one of the latter spaces from a Bernstein space. It will be used for estimating remainders and studying rates of convergence.

In the main part, we present the desired extensions. Our applications include the classical Whittaker-Kotel'nikov-Shannon sampling formula, the reproducing kernel formula, the Parseval decomposition formula, Bernstein's inequality for derivatives, and Nikol'ski{\u{\i}}'s inequality estimating the  $l^p(\Z)$ norm in terms of the $L^p(\R)$ norm.
\end{abstract}

\begin{keyword}
Non-bandlimited functions \sep Sobolev spaces of fractional order \sep modulation spaces \sep Lipschitz spaces \sep Riesz derivatives of fractional order \sep formulae with remainders \sep derivative-free error estimates \sep sampling formulae \sep reproducing kernel formula \sep Parseval decomposition formula \sep Bernstein's inequality


\MSC[2010] 41A17 \sep 41A80 \sep 42A38 \sep 46E15 \sep 94A20 \sep 26A16 \sep 46E35 \sep 26A33

\end{keyword}

\end{frontmatter}

\section{Overview}
A main subject of this paper is the re-adapted modulation space $\Mn(\R)=\Mn$, which is based on the classical modulation space $M^{2,1}(\R)$ introduced by Feichtinger \cite{,Feichtinger_1983a,Feichtinger_1983}.
The new space $\Mn$ comprises of all functions $f\in \Ltwo$ with norm
\[
  \li\|f\ri\|_{\Mn}:=\norm{f}_{\Ltwo}+\sup_{0<h\le 1} \sum_{n\in\Z\setminus\{-1, 0\}}\sm{} <\infty.
\]
Comparing this norm with that of the classical modulation space $M^{2,1}$ (see \eqref{eq_classical_modulation_norm} below) shows that $\Mn\subset M^{2,1}$, the fact that it is a proper subspace will be shown.

Another space of basic importance in this paper will turn out to be the fractional Sobolev space, also called Bessel potential space or Liouville space, of order $\alpha>0$, namely,
\begin{equation}\label{eq_W_alpha}
  \Rieszalpha:=\big\{f\in\Ltwo\;:\;\exists\ g\in\Ltwo \text{ with } \wh g (v)=|v|^\alpha \wh{f}(v)\}\qquad( \alpha>0).
\end{equation}
It is associated with a fractional order derivative of $f\in\Ltwo$, namely the strong, normed Riesz derivative $D^{\{\alpha\}}\kern-2pt f$, defined for $0< \alpha < 2j$, $j\in\N$, by
\[
  \lim_{\eps\to 0+}\Biggnorm{\frac1{C_{\alpha,2j}} \int_\varepsilon^\infty \frac{\overline\Delta^{2j}_u f(\,\cdot\,)}{u^{1+\alpha}}\du - D^{\{\alpha\}}\kern-2pt f(\,\cdot\,)}_{\Ltwo}=0.
\]
for a specific constant $C_{\alpha,2j}$, the difference $\overline\Delta$ being the central one; see \eqref{eq_C}, \eqref{eq_difference_central}.
It will turn out that $\Rieszalpha$ can be characterized as
\[
  \Rieszalpha = \big\{f\in L^2(\R)\;:\; D^{\{\alpha\}} f\in L^2(\R)\big\};
\]
see Proposition~\ref{prop_Riesz_existence_equiv}.

One of the essential results of this paper is that the space $\Mn$ lies between two Lipschitz spaces, the left one being of order $\alpha$ for any $\alpha>1/2$, the right one being the specific $\Lip_r(\frac12)$ of order $1/2$ (see Section~\ref{sec_Lipschitz} for the definition). Further, $\Lip_r(\alpha)$ lies between $\Rieszalpha$ and $\Rieszbeta$ for any $0<\beta <\alpha <r$. Formally this result reads that for $0<\beta < \frac12 < \alpha < r$, and any $r\in\N$,
\begin{equation}\label{eq_inclusion_intro}
\Rieszalpha\cap C(\R) \subsetneqq \Lip_r(\alpha)\cap C(\R) \subsetneqq \Mn \subsetneqq \Lip_r\pBig{\frac12}\cap C(\R)\subsetneqq \Rieszbeta\cap C(\R).
\end{equation}

One goal of this paper is to show that the readapted modulation space $\Mn$ not only has theoretical applications but especially also those of a more practical nature.

A fundamental inequality in analysis is Bernstein's inequality (see Section~\ref{bernst_ineq}). For functions belonging to the Bernstein space $ B^2_\sigma$, it reads
\[
  \bignorm{f^{(s)}}_{\Ltwo}\le \sigma^s\bignorm{f}_{\Ltwo}\qquad (f\in B^2_\sigma).
\]

If $f$ belongs to the Sobolev space $W^{s,2}(\R)\cap C(\R)$ (see Section~\ref{sec_Sobolev}) for some $s\in\N$ with $v^s\wh f(v)\in \Lone$, rather than to the smaller space $B^2_\sigma$, then
\[
  \bignorm{f^{(s)}}_{\Ltwo}\le \sigma^s\bignorm{f}_{\Ltwo}+\dist_2\pbig{f^{(s)},B^2_\sigma}\qquad (\sigma>0),
\]
where the remainder $\dist_2$ is given by (see Section~\ref{distance} for details)
\[
  \dist_2\pbig{f^{(s)},B^2_\sigma} =\brbigg{\int_{\abs{v}>\sigma}\bigabs{v^s\wh f(v)}^2\dv}^{1/2}.
\]

Concerning the behaviour of this remainder, the following assertions will be shown to be equivalent for $0\le \beta < \alpha < r$ and each $s\in\N_0$ with $s< \alpha$,

\begin{aufzaehlung}{$(ii)$}
\item $\ds f^{(s)} \in \Lip_r(\alpha)$,
\item $\ds \dist_2\pbig{f^{(s)},B^2_\sigma} = \Oh\pbig{\sigma^{-\alpha}}\quad (\sigma\to\un)$.
\end{aufzaehlung}

Recalling \eqref{eq_inclusion_intro}, we see that the space $\Mn$ lies between two Lipschitz spaces, namely,
\begin{equation}\label{eq_inclusion_intro_2}
    \Lip_r(\alpha)\cap C(\R) \subsetneqq \Mn \subsetneqq \Lip_r\pBig{\frac12}\cap C(\R)\qquad (\alpha>1/2).
\end{equation}
It follows from the right-hand inclusion in \eqref{eq_inclusion_intro_2} and (i)$\,\Leftrightarrow\,$(ii) above that $f^{(s)}\in \Mn$ yields the estimate
\begin{equation}\label{eq_modulation_order_intro}
  \dist_2\pbig{f^{(s)},B^2_\sigma} = \Oh\pbig{\sigma^{-1/2}}\quad (\sigma\to\un).
\end{equation}
On the other hand, the left-hand inclusion relation in \eqref{eq_inclusion_intro_2} shows that the order in \eqref{eq_modulation_order_intro} cannot be improved to $\Oh\pbig{\sigma^{-1/2-\eps}}$ for any $\eps > 0$ arbitrarily small.
%On the other hand, the order $\Oh\pbig{\sigma^{-1/2}}$ is the same as that for functions $f\in\Lip_r\!\big(\hal\big)$. This corresponds to the second inclusion %relation from the right in \eqref{eq_inclusion_intro}.
Nevertheless, the question may arise whether it might be possible to improve the order in \eqref{eq_modulation_order_intro} to $\oh\pbig{\sigma^{-1/2}}$. The answer is \emph{no} as will be seen in Proposition~\ref{prop_oh}.

If one, however, replaces the space $\Mn$ by $\Rieszhal$, then
\[
  f^{(s)}\in \Rieszhal \implies \dist_2\pbig{f^{(s)},B^2_\sigma} = \oh\pbig{\sigma^{-1/2}}\quad (\sigma\to\un);
\]
see Corollary~\ref{bernst_Sob} for the details.

\section{Some notations}\label{notation}
For $p\in [1, \infty]$ and $f\in L^p(\R)$, we define
\[
    \|f\|_{L^p(\R)} := \brbigg{\int_\R \abs{f(u)}^p\du}^{1/p}\quad   (1\le p<\infty)
\]
with the usual modification for $p=\infty$. By $C(\R)$ we denote the class of all functions $f\colon\R \to \C$ that are continuous on $\R$.

For the Fourier transform $\wh{f}$ of a function $f$ we prefer the normalization
\[
   \wh{f}(v):=\rez{\sqrt{2\pi}} \,\int_\R f(u) e^{-iuv}\,du\qquad(v\in\R).
\]
For $f\in L^p(\R)$, the integral exists as an ordinary Lebesgue integral when $p=1$ while for $p\in (1,2]$ it is
defined by a limiting process; see \cite[\S\S\,5.2.1--5.2.2]{Butzer-Nessel_1971}.

For $\sigma>0$, let $B_\sigma^2$ be the \emph{Bernstein space} or \emph{Paley-Wiener space} comprising all functions $f\in\Ltwo$, the Fourier transform of which vanishes outside $\bk{-\sigma,\sigma}$. The most prominent example of a function in $B^2_\pi$ is the $\sinc$ function, given by
\begin{align*}
    \sinc z :=
    \begin{cases}
        \dfrac{\sin(\pi z)}{\pi z}, & z\in\C\setminus\{0\},\\[2ex]
        1, & z=0,
    \end{cases}\qquad
     \wh\sinc(v)= \rez{\sqrt{2\pi}}\cdot
     \begin{cases}
             1,      & \abs{v}<\pi,\\[1ex]
            \rez{2}, & \abs{v}=\pi,\\[1ex]
             0,      & \abs{v}>\pi.
     \end{cases}
     %\rez{\sqrt{2\pi}}\, \rect(v)
     %\qquad(v\in\R).
\end{align*}

\section{A hierarchy of spaces extending Bernstein spaces; fractional order derivatives}\label{hierarchy}
The membership of $f$ in $B_\sigma^2$ has many important consequences such as the continuity of $f$, the existence of a Fourier transform $\wh{f}$ belonging to $L^1(\R)$, the reconstruction of $f$ from its Fourier transform and the $\ell^2(\Z)$ summability of samples. Thus, when one looks for suitable generalizations of the Bernstein space $B_\sigma^2$, it is desirable to preserve these properties.

\subsection{Fourier inversion classes}\label{sec_Fourier_inversion_classes}

In order to extend the Bernstein space $B_\sigma^2$ to larger function spaces, we weaken the property of $\wh f$ vanishing outside the compact interval $\bk{-\sigma,\sigma}$, to $\wh f$ belonging to $L^1(\R)$. This still guarantees the reconstructibility of $f$ from its Fourier transform in terms of the inversion formula
\begin{equation}\label{eq_Fourier_inversion}
   f(t) = \rez{\sqrt{2\pi}} \int_\R \wh{f}(v) e^{ivt}\,dv \qquad (t\in\R);
\end{equation}
see \cite[Prop.~5.1.10, 5.2.16]{Butzer-Nessel_1971}.
More generally, we introduce the \emph{Fourier inversion classes} (cf.\ \cite{Butzer-Gessinger_1997, Butzer-Higgins-Stens_2005, Butzer-Splettstoesser-Stens_1988, Butzer-Schmeisser-Stens_2012}),
\[
     F^{s,2}:=\brbig{f\in L^2(\R)\cap C(\R)\,:\, v^s\wh f(v)\in L^1(\R)} \qquad (s\in \N_0).
\]
For $s=0$ we simply write $F^2$ instead of $F^{0,2}$. If $0\le s_1 \le s_2$, then there holds $F^{s_2,2}\subset F^{s_1,2}\subset F^2$. In addition to \eqref{eq_Fourier_inversion}, one has for $f\in F^{s,2}$ that the derivative $f^{(s)}$ exists, belongs to $C(\R)$ and has the
representation
\begin{equation}\label{eq_derivative_representation}
    f^{(s)}(t) = \rez{\sqrt{2\pi}} \int_\R (iv)^s\wh{f}(v) e^{ivt}\,dv \qquad (t\in\R);
\end{equation}
see \cite[Proposition~5.1.17 with $f$ replaced by $\wh{f}\,$]{Butzer-Nessel_1971}.

However, other than in $B_\sigma^2$, the membership of $f$ in $F^{s,2}$ does not guarantee the summability of samples of $f$. Therefore, whenever samples for uniformly spaced points such as $h\Z$ are involved, we shall need in addition that $f$ belongs to
\[
     S_h^p\,:=\,\brbig{\strut f\,:\,\R\to\C\;:\; \li(\strut f(hk)\ri)_{k\in\Z}\in \ell^p(\Z)}\qquad(p=1,2).
\]
We may call $S_h^p$ the {\it $\ell^p$ summability class} for \emph{step size} $h$. Note that $S_h^{1}\subset S_h^{2}$. Furthermore, if $f\in B_\sigma^2$ for some $\sigma>0$, then
\[
    f \in F^{s,2} \cap S_h^2
\]
for every $s\in\N_0$ and $h>0$. %, where $f|_\R$ denotes the restriction of $f$ to $\R$.
We may therefore consider $F^{s,2}\cap S_h^2$ as well as $F^2$ itself as extensions of $B_\sigma^2$.

The Fourier inversion classes are in some sense the most general spaces in which our studies can be performed. Spaces between $B_\sigma^2$ and $F^{s,2}$ are also of interest since they will yield smaller errors in the extended formulae.

\subsection{Sobolev spaces}\label{sec_Sobolev}
For $r\in\N$, denote by $\mathrm{AC}_{{\rm loc}}^{r-1}(\R)$ the class of all functions that are $(r-1)$-times locally absolutely continuous on $\R$; see \cite[pp.~6--7]{Butzer-Nessel_1971}. The following class has been considered in Fourier analysis:
\begin{equation*}%\label{def_Sobolev}
W^{r,2}(\R) := \Big\{f\,:\,\R\to\C\;:\; f=\phi \text{ a.\,e., }
    \phi\in \mathrm{AC}_{{\rm loc}}^{r-1}(\R), \, \phi^{(k)}\in \Ltwo,\ 0\le k\le r\Big\};
\end{equation*}
see \cite[(3.1.48)]{Butzer-Nessel_1971}. For $f\in W^{r,2}(\R)$, we may write $f^{(k)}$ instead of $\phi^{(k)}$ for $k=0,\dots, r$. By endowing $W^{r,2}(\R)$ with the norm
\begin{equation}\label{eq_Sobolev_norm}
  \|f\|_{W^{r,2}(\R)} :=  \bigg\{\sum_{k=0}^r \bignorm{f^{(k)}}_{\Ltwo}^2\bigg\}^{1/2},
\end{equation}
we may identify it as a \emph{Sobolev space.} In connection with Fourier transforms, the following alternative description of $W^{r,2}(\R)$ is of interest; see
\cite[Theorem~5.2.21]{Butzer-Nessel_1971}.

\begin{proposition}\label{propw1}
We have
\begin{align*}
   \Wrtwo ={}& \brBig{f\in\Ltwo\;:\; v^r \wh{f}(v)\in\Ltwo}\\[2ex]
   {}={}&\Big\{f\in\Ltwo\;:\; (iv)^r \wh{f}(v)=\wh g(v),\ g\in\Ltwo\Big\}.
\end{align*}
Furthermore, $\wh{f^{(r)}}(v)=(iv)^r\wh f(v)=\wh g(v)$ a.\,e.
\end{proposition}
The two characterizations of the space $\Wrtwo$ coincide since the Fourier transform is an isometry from $\Ltwo$ onto itself.

An important inequality in analysis is that of S.\,M.~Nikol'ski\u{\i} (1951) given in \eqref{nikol0} below. From the proof in \cite[pp.~123--124]{Nikolskii-1975} we can extract the following statement.
\begin{proposition}\label{propw2}
Let $f\in W^{1,2} \cap C(\R)$. Then
\[
  \bigg\{ h \sum_{k\in\Z} \abs{f(hk)}^2 \bigg\}^{1/2} \le \|f\|_{\Ltwo}+  h\|f'\|_{\Ltwo}
\]
for any $h>0.$
\end{proposition}

Propositions~\ref{propw1} and \ref{propw2} imply that for $r\in \N$ and $h>0$,
\begin{equation}\label{sobolev1}
  B_\sigma^2 \subsetneqq\Wrtwo \cap C(\R) \subsetneqq F^{r-1,2}\cap S_h^2\subset F^2 \cap S_h^2\subsetneqq F^{2} \subsetneqq L^2(\R).
\end{equation}

\subsection{Fractional order derivatives}\label{sec_frac_der}
In order to generalize Proposition~\ref{propw1} to fractional order derivatives, we consider the spaces
\begin{align*}
   \Rieszalpha:= {}& \brBig{f\in\Ltwo\;:\; |v|^\alpha \wh{f}(v)\in\Ltwo}\\[2ex]
   {}={}&\Big\{f\in\Ltwo\;:\; |v|^\alpha \wh{f}(v)=\wh g(v),\ g\in\Ltwo\Big\}\qquad (\alpha>0).
\end{align*}

In view of Proposition~\ref{propw1} there holds
\begin{equation}\label{eq_W_r}
  \Rieszr=\Wrtwo \qquad (r\in\N).
\end{equation}

In this section we are going to characterize the spaces $\Rieszalpha$ for arbitrary $\alpha>0$ in terms of fractional order derivatives.
 %To this end we need further information concerning fractional order derivatives.

For $\alpha>0$, $j\in\N$ with $2j>\alpha$ we set
\[
  {\Rae f}(x):=\frac{1}{C_{\alpha,2j}}\int_\varepsilon^\infty \frac{\overline\Delta^{2j}_u f(x)}{u^{1+\alpha}}\du\qquad(x\in\R),
\]
where
\begin{equation}\label{eq_C}
  C_{\alpha,2j}:=(-1)^j2^{2j-\alpha} \int_0^\infty \frac{\sin^{2j}u}{u^{1+\alpha}}\,du,
\end{equation}
and
\begin{equation}\label{eq_difference_central}
  \overline\Delta^{2j}_u f(x):=\sum_{k=0}^{2j}(-1)^k\bin {2j} k f\pbig{x+\p{j-k}u}\qquad(x,u\in\R)
\end{equation}
is the central difference of $f$ of order $2j$ at $x$ with increment $u$.
$\Rae$ turns out to be a bounded linear operator mapping $\Ltwo$ into itself satisfying
\begin{equation}\label{eq_Rae_bound}
  \bignorm{\Rae f}_{\Ltwo}\le \frac{2^{2j}}{\alpha\eps^\alpha\abs{C_{\alpha,2j}}}\norm{f}_{\Ltwo}\qquad \pbig{f\in\Ltwo}.
\end{equation}

\begin{proposition}\label{prop_Rae}
The Fourier transform of $\Rae f$ is given by
\begin{equation}\label{eq_FT_Rae}
   \wh{\Rae f}(v)= \eta_{2j,\alpha,\eps}(v)\wh f(v)\quad a.\,e.,
\end{equation}
where
\[
  \eta_{2j,\alpha,\eps}(v):=\frac{(-1)^j 2^{2j}}{C_{\alpha,2j}}\int_\eps^\un u^{-1-\alpha} \sin^{2j}\pBig{\frac{vu}{2}}\du\qquad\p{v\in\R}.
\]
\end{proposition}

Before proving Proposition~\ref{prop_Rae}, we list three properties of the function $\eta_{2j,\eps,\alpha}$, the proofs of which are quite elementary.

\begin{lemma}\label{la_eta_je}
For $\eta_{2j,\alpha,\eps}$, as defined above, there holds
\begin{aufzaehlung}{$(iii)$}
\item $\abs{\eta_{2j,\alpha,\eps}(v)}\le M_{2j,\alpha,\eps}\qquad (v\in\R)$\\[1ex]
for some constant $M_{2j,\alpha,\eps}$, independent of $v$,
\item $\abs{\eta_{2j,\alpha,\eps}(v)}\le \abs{v}^\alpha \qquad (v\in\R;\eps>0)$,
\item $\ds \lim_{\eps\to 0+} \eta_{2j,\alpha,\eps}(v)= \abs{v}^\alpha\qquad (v\in\R)$.
\end{aufzaehlung}
\end{lemma}

Now to the proof of Proposition~\ref{prop_Rae}:

\begin{proof}
In case $f\in \Lone\cap\Ltwo$ the conclusion of Proposition~\ref{prop_Rae} follows easily by Fubini's theorem (see \cite[top of p.~414]{Butzer-Nessel_1971} for $j=1$). For arbitrary $f\in\Ltwo$ we choose a sequence $f_n\in \Lone \cap\Ltwo$, $n\in\N$ with $\lim_{n\to\un}\norm{f_n-f}_{\Ltwo}=0$. Then by \eqref{eq_Rae_bound}, Lemma~\ref{la_eta_je}\,(i), and the isometry property of the Fourier transform,
\begin{align*}
  \Bignorm{\wh{\Rae f}- \eta_{2j,\alpha,\eps}\wh f\,}_{\Ltwo}{}\le{}&
  \Bignorm{\wh{\Rae f}- \wh{\Rae f_n}}_{\Ltwo} + \bignorm{\eta_{2j,\alpha,\eps}\bigbk{\wh f_n-\wh f\;}}_{\Ltwo}\\[2ex]
  {}\le{}& M \norm{f-f_n}_{\Ltwo} + M_\eps \norm{f_n-f}_{\Ltwo}=\oh(1)\quad (n\to\un),
\end{align*}
where $M:= {2^{2j}}\p{\alpha\eps^\alpha|C_{\alpha,2j}|}^{-1}$ is the constant on the right-hand side of \eqref{eq_Rae_bound} and $M_\eps:=M_{2j,\alpha,\eps}$ is that in Lemma~\ref{la_eta_je}\,(i).
This proves the assertion.
\end{proof}

\begin{definition}\label{def_Riesz_derivative}
A function $f\in L^2(\R)$ is said to have a strong (norm) Riesz derivative of fractional order $0<\alpha < 2j$, $j\in \N$, if there exists function $g\in L^2(\R)$ such that
\[
\lim_{\eps\to 0+}\bignorm{\Rae f-g}_{\Ltwo}=
  \lim_{\eps\to 0+}\Biggnorm{\frac1{C_{\alpha,2j}} \int_\varepsilon^\infty \frac{\overline\Delta^{2j}_u f(\,\cdot\,)}{u^{1+\alpha}}\du-g(\,\cdot\,) }_{\Ltwo}=0.
\]
Then the strong Riesz derivative is defined by $D^{\{\alpha\}}\kern-2pt f:=g$.
\end{definition}

Of course one has to show that Definition~\ref{def_Riesz_derivative} is independent of $j\in\N$, which is implicitly contained in Proposition~\ref{prop_FT_Riesz_derivative} below.

Let us observe that in case $0\le \alpha < 2$ one may choose $j=1$ and the operator $R^{\{\alpha\}}_{2,\eps}$ can simply be rewritten as
\[
  R^{\{\alpha\}}_{2,\eps}\kern-1pt f(x)
=\frac1{C_{\alpha,2}}\int_\varepsilon^\infty \frac{\overline\Delta^{2}_u f(x)}{u^{1+\alpha}}\du
  =\frac{1}{\Lambda_c(-\alpha)}\int_{\abs u\ge \eps}\frac{f(x-u)-f(x)}{\abs u ^{1+\alpha}}\du
\]
with
\[
  \Lambda_c(\alpha):=2\Gamma(\alpha)\cos\pBig{\frac{\pi\alpha}{2}},
\]
where the singularity of $\Lambda_c(\alpha)$ at $\alpha=-1$ is removed by setting $\Lambda_c(-1)=-\pi$. This case is treated for $1\le p\le 2$ in great detail in \cite[Section~11.3]{Butzer-Nessel_1971}, the extension to arbitrary $\alpha >0$ being straightforward. The proofs presented here are much simpler since the matter is restricted to $p=2$.

In this instance, Definition~\ref{def_Riesz_derivative} turns out to be the \emph{classical} fractional order derivative studied by M.~Riesz in his innovative treatise of 1927 \cite{Riesz_1927}. The article of A.~Marchaud \cite{Marchaud_1927}, also of that year, which plays an essential role in approximation theory and fractional calculus, is the basis of the extension to arbitrary $\alpha>0$. See also \cite{Samko_2002,Samko-Kilbas-Marichev_1993}, and for a modulus of smoothness related to the Riesz derivative see \cite{Runovski-Schmeisser_2014}.

\begin{proposition}\label{prop_FT_Riesz_derivative}
If $f\in \Ltwo$ has a strong Riesz derivative of order $\alpha>0$, then
\begin{equation}\label{eq_FT_Riesz}
   \wh{D^{\{\alpha\}}\kern-2pt f}(v)= |v|^\alpha\wh f(v)\quad a.\,e.
\end{equation}
\end{proposition}
\begin{proof}
By Proposition~\ref{prop_Rae} and Lemma~\ref{la_eta_je}\,(iii) we have
\[
  \lim_{\eps\to 0+} \wh{\Rae f} (v)=\lim_{\eps\to 0+} \eta_{2j,\alpha,\eps}(v)\wh f (v) =\abs{v}^\alpha \wh f (v)\qquad (v\in\R).
\]

On the other hand, by the isometry property of the Fourier transform,
\[
  \lim_{\eps\to 0+}\Bignorm{\wh{\Rae f} -\wh{D^{\{\alpha\}}\kern-2pt f} }_{\Ltwo}
  =\lim_{\eps\to 0+}\bignorm{\Rae f-D^{\{\alpha\}}\kern-2pt f }_{\Ltwo}=0.
\]
Since the pointwise limit must coincide a.\,e.\ with the strong limit, the assertion follows.
\end{proof}

Noting Proposition~\ref{prop_FT_Riesz_derivative}, we see that the Riesz derivative may equivalently be defined in terms of the inverse Fourier transform by
\begin{equation}\label{eq_riesz_equiv_defi}
   D^{\{\alpha\}}\kern-2pt f(x)= \frac{1}{\sqrt{2\pi}}\int_\R|v|^\alpha\wh f(v)e^{ivx}\dv,
\end{equation}
where the convergence of the integral is to be understood in $\Ltwo$-norm. In particular, if $\abs{v}^\alpha\wh f(v)\in\Lone$, then the integral in \eqref{eq_riesz_equiv_defi} exists as an ordinary Lebesgue integral.

Now to the characterization of the space $\Rieszalpha$ in terms of fractional order derivatives.
\begin{proposition}\label{prop_Riesz_existence_equiv}
The following assertions are equivalent for $f\in\Ltwo$:
\begin{aufzaehlung}{$(iii)$}
\item $f$ has a strong Riesz derivative $D^{\{\alpha\}}\kern-2pt f$;
\item there holds
\[
  \bignorm{{\Rae f}(v)}_{\Ltwo}=
  \Bigg\|\frac1{C_{\alpha,2j}} \int_\varepsilon^\infty \frac{\overline\Delta^{2j}_u f(\,\cdot\,)}{u^{1+\alpha}}\du \Bigg\|_{L^2(\R)}=\Oh(1)\qquad (\eps\to0+);
\]
\item $f\in \Rieszalpha$, \ie\ $|v|^\alpha \wh f(v)\in \Ltwo$.
\end{aufzaehlung}
In this event, the function $g\in \Ltwo$ defined via $\wh g(v)= |v|^\alpha \wh f(v)$ a.\,e.\ on $\R$  is just the derivative  $D^{\{\alpha\}}\kern-2pt f$.
\end{proposition}

\begin{proof} The implication (i)$\, \Rightarrow\,$(ii) is obvious. Now, let (ii) by satisfied. Noting Lemma~\ref{la_eta_je} (iii) and Proposition~\ref{prop_Rae}, we have by Fatou's lemma that
\begin{align*}
  &{}\bignorm{\abs{v}^\alpha\wh f(v)}_{\Ltwo}\le \liminf_{\eps \to 0+}\bignorm{\eta_{2j,\alpha, \eps}(v)\wh f(v)}_{\Ltwo}\\[2ex]
  {}={}& \liminf_{\eps \to 0+}\Bignorm{\wh{\Rae f}(v)}_{\Ltwo}=\liminf_{\eps \to 0+}\bignorm{{\Rae f}(v)}_{\Ltwo}.
\end{align*}
Since the latter term is finite by assumption, there follows (iii).

In order to prove the implication (iii)$\,\Rightarrow\,$(i), assume that $f\in \Rieszalpha$. The surjectivity of the Fourier transform yields $|v|^\alpha \wh f(v)=\wh g(v)$ for some $g\in \Ltwo$, and
\[
  \bignorm{{\Rae f}(v)-g(v)}_{\Ltwo}=\Bignorm{\wh{\Rae f}(v)-\wh g(v)}_{\Ltwo}= \Bignorm{\bkbig{\eta_{2j,\alpha,\eps}(v)-\abs{v}^\alpha}\wh f(v)}_{\Ltwo}.
\]
By Lemma~\ref{la_eta_je}\,(ii) we have
\[
  \abs{\eta_{2j,\alpha,\eps}(v)-\abs{v}^\alpha}^2\bigabs{\wh f(v)}^2\le 4\abs{v}^{2\alpha}\bigabs{\wh f(v)}^2 = 4\abs{\wh g(v)}^2\in \Lone,
\]
and hence in view of Lebesgue's dominated convergence theorem and Lemma~\ref{la_eta_je}\,(iii),
\[
  \lim_{\eps\to 0+}\Bignorm{\bkbig{\eta_{2j,\alpha,\eps}(v)-\abs{v}^\alpha}\wh f(v)}_{\Ltwo}
     =\Bignorm{\lim_{\eps\to 0+}\bkbig{\eta_{2j,\alpha,\eps}(v)-\abs{v}^\alpha}\wh f(v)}_{\Ltwo}=0.
\]
It follows that
\[
  \lim_{\eps\to 0+}\bignorm{{\Rae f}(v)-g(v)}_{\Ltwo}=0,
\]
\ie, $g$ is the strong Riesz derivative of order $\alpha$ of $f$.
\end{proof}

Observe that $\Rieszalpha$ is a normalized Banach space under the norm (see \cite[pp.~373, 381]{Butzer-Nessel_1971})
\[
  \norm{f}_{\Rieszalpha}:=\norm{f}_{\Ltwo}+\norm{g}_{\Ltwo}=\norm{f}_{\Ltwo}+\bignorm{D^{\{\alpha\}}\kern-2pt f}_{\Ltwo}.
\]

We have already seen in \eqref{eq_W_r} that the classes $\Wrtwo$ and $\Rieszr$ coincide for $r\in\N$. Furthermore, one has for $f\in W^{r,2}(\R)= \Rieszr$, in view of Propositions~\ref{propw1} and~\ref{prop_Riesz_existence_equiv} that
\[
  \bignorm{f^{(r)}}_{\Ltwo}= \bignorm{D^{\{r\}}\kern-2pt f}_{\Ltwo} = \bignorm{v^r \wh f (v)}_{\Ltwo}\qquad (r\in\N).
\]
This means that
\[
  \norm{f}_{\Rieszr}=\norm{f}_{\Ltwo}+ \bignorm{D^{\{r\}}\kern-2pt f}_{\Ltwo}= \norm{f}_{\Ltwo}+ \bignorm{f^{(r)}}_{\Ltwo},
\]
where the latter expression defines a norm on $W^{r,2}(\R)$, which is equivalent to the norm \eqref{eq_Sobolev_norm}; see \cite[p.~242]{Edmunds-Evans_1987}, \cite[Section~5.1]{Adams-Fournier_2003}. Hence the spaces $W^{r,2}(\R)$ and $\Rieszr$ are equal with equivalent norms.

As to the ordinary derivative $f^{(r)}$ and the Riesz derivative $D^{\br{r}}\kern-2pt f$ we have for their respective Fourier transforms
\begin{align*}
\wh{f^{(r)}}(v)={}& (iv)^r \wh f(v)\quad a.\,e.,\qquad
\wh{D^{\{r\}}\kern-2pt f}(v)= |v|^r\wh f(v) \quad a.\,e.
\end{align*}
It follows  that for even $r=2m$,
\[
  D^{\{2m\}}\kern-2pt f(x)=(-1)^m f^{(2m)}(x)\quad a.\,e.
\]
For odd $r=2m-1$ we have to make use of the Hilbert transform $\widetilde f$, having Fourier transform
\[
  %[\widetilde f\,]\wh{\phantom{f}}(v) \wh{f\kern2pt\widetilde{\phantom{g}}}(v)\quad \wh{\mbox{$\ds\widetilde f$\,\,}}\!\!(v)\qquad
  \wh{\widetilde f\kern3pt}\kern-3pt(v) =(-i\sgn v)\wh f (v)\quad a.e.
\]
This yields (see also \cite[p.~406]{Butzer-Nessel_1971})
\[
  D^{\{2m-1\}}f(x)
  %{f\kern1pt\widetilde{\phantom{g}}}(x)
  =(-1)^{m-1}{\widetilde f\,}^{(2m-1)}(x)= (-1)^{m-1}\widetilde{{ f^{(2m-1)}}}(x)\quad a.\,e.
\]

There exists an alternative approach to strong derivatives of integer order $r\in\N$. In this approach the role of the central difference in the definition of the Riesz derivative is clarified. A function $f\in\Ltwo$ is said to have a strong Riemann derivative of order $r\in\N$, or an $r$-th order Riemann derivative in $\Ltwo$, if there exists a function $g\in\Ltwo$ such that
\[
  \lim_{h\to 0}\biggnorm{ \frac{\overline\Delta^{r}_h}{h^r} f-g }_{\Ltwo}=0.
\]
Then the strong Riemann derivative is defined by $D^{[r]}\kern-2pt f:=g$. It is known that $f$ has an $r$-th order Riemann derivative if and only if $f\in\Wrtwo$; see \cite[pp.~385--386]{Butzer-Nessel_1971}. In this event one has $D^{[r]}\kern-2pt f=f^{(r)}$. Moreover the following five assertions are equivalent:
\begin{aufzaehlung}{$(iv)$}
\item $\ds f\in \Wrtwo$,
\item $\ds \widetilde f\in \Wrtwo$,
\item $\ds \biggnorm{\frac{\overline\Delta^{r}_h f}{h^r}}_{\Ltwo}=\Oh(1)\quad (h\to0)$,
\item $\ds \biggnorm{\frac{\overline\Delta^{r}_h \widetilde f}{h^r}}_{\Ltwo}=\Oh(1)\quad (h\to0)$,
\item there exists $g\in \Ltwo$ such that
\begin{align*}
  \ds f(x)={}&\int_{-\infty}^x du_1 \int_{-\infty}^{u_1} du_2\cdots \int_{-\infty}^{u_{r-2}}du_{r-1}
   \int_{-\infty}^{u_{r-1}}g(u_r)\du_r\\[2ex]
  ={}& \frac1{(r-1)!}\int_{-\infty}^x \frac{g(u)}{(x-u)^{1-r}}\du \quad a.\,e.,
\end{align*}
where each of the iterated integrals exists only conditionally as a function in $\Ltwo$; see \cite[p.~226]{Butzer-Nessel_1971}.
\end{aufzaehlung}

Now to the identification of the Riesz derivative with the Riemann derivative. Since
\begin{align*}
\wh{D^{\bk{r}}\kern-2pt f}(v)={}& (iv)^r \wh f(v)\quad a.\,e.,\qquad
\wh{D^{\{r\}}\kern-2pt f}(v)= |v|^r\wh f(v) \quad a.\,e.
\end{align*}
it follows that $f$ has a Riesz derivative $D^{\br{r}}\kern-2pt f $ if and only if it has a Riemann derivative $D^{\bk{r}}\kern-2pt f$, and there holds a.\,e.,
\begin{align*}
  D^{\{r\}}\kern-2pt f(x)={}&
  \begin{cases}
    (-1)^m D^{\bk{2m}}\kern-2pt f(x),& r=2m,\\[2ex]
    (-1)^{m-1}D^{\bk{2m-1}}\kern-2pt f(x)= (-1)^{m-1}\widetilde{D^{\bk{2m-1}}\kern-2pt f}(x),& r=2m-1,
  \end{cases}\\[3ex]
 ={}&
 \begin{cases}
    (-1)^m f^{(2m)}(x), & r=2m,\\[2ex]
    (-1)^{m-1}\bkbig{\widetilde f\,}^{(2m-1)}(x)= (-1)^{m-1}\widetilde{{ f^{(2m-1)}}}(x),& r=2m-1.
  \end{cases}
\end {align*}

\subsection{Lipschitz spaces}\label{sec_Lipschitz}

The modulus of smoothness of $f\in \Ltwo$ of order $r\in\N$ is defined by
\[
        \omega_r(f;\delta;\Ltwo):= \sup_{|h|\le \delta} \|\Delta_h^r f\|_{\Ltwo}\qquad(\delta>0),
\]
where
\begin{equation*}%\label{eq_difference}
       (\Delta^r_h f)(x):= \sum_{j=0}^r (-1)^{r-j} \bin{r}{j} f(x+jh)
\end{equation*}
is the forward difference of order $r$ at $x$ with increment $h$.
Some basic properties are
\begin{equation*}%\label{eq_omega_mit_faktor}
        \omega_r(f;\lambda\delta ;\Ltwo) \le (1+\lambda)^r\omega_r(f;\delta ;\Ltwo)\qquad(\lambda,\delta>0),
\end{equation*}
and for $f\in W^{s,2}(\R)$ and any $j\in\N$,
\begin{align}\label{modulus_mit_ableitung}
\omega_{s+j}(f;\delta;\Ltwo) &{}\le \delta^s \omega_j(f^{(s)}; \delta;\Ltwo)\qquad(\delta>0),\\[2ex]
\omega_{s}(f;\delta;\Ltwo) &{} \le \delta^s \bignorm{f^{(s)}}_{\Ltwo}\qquad(\delta>0);\label{modulus_mit_ableitungsnorm}
\end{align}
see e.\,g.~\cite[Chap.~2, {\S}\,7]{DeVore-Lorentz_1993}.

The Lipschitz classes based on the modulus $\omega_r$ of order $\alpha$, $0<\alpha\le r$, are defined by
\[
  \Lip_r(\alpha)=\Lip_r(\alpha;\Ltwo)
   :=\big\{f\in \Ltwo\;:\;\omega_{r}(f;\delta;\Ltwo)=\Oh(\delta^\alpha),\ \delta\to0+\big\}.
\]
$\Lip_r(\alpha)$ is a normalized Banach space under the norm (see \cite[pp.~373, 376]{Butzer-Nessel_1971})
\begin{equation}\label{eq_Lip_norm}
  \norm{f}_{\Lip_r(\alpha)}:=\norm{f}_{\Ltwo} + \sup_{h>0}\brbig{h^{-\alpha}\bignorm{\Delta^r_h f }_{\Ltwo}}.
\end{equation}
One should observe that $\Lip_r(\alpha)$ is nothing but the particular Besov space  $B^\alpha_{2\infty}$. Indeed, $B^\alpha_{2\infty}$ can be defined as the set of those $f\in\Ltwo$ with
%
%\begin{equation}\label{eq_Besov_seminorm}
%%
% \norm{f}_{B^\alpha_{2\infty}}:= \sup_{k\in\N}2^{\alpha k}\brbigg{\int_{2^{k-1}\le\abs{v}<2^{k}}\bigabs{\wh f(v)}^2\dv}^{1/2}<\infty;
%%
%\end{equation}
%
\begin{equation}\label{eq_Besov_norm}
 \norm{f}_{B^\alpha_{2\infty}}:= \sup_{k\in\N_0}2^{\alpha k}\brbigg{\int_{\R} \chi_k(v)\bigabs{\wh f(v)}^2\dv}^{1/2}<\infty,
\end{equation}
where $\chi_0$ is the characteristic function of the interval $[-1,1]$ and $\chi_k$, $k\in\N$, that of the set $\br{v\in\R\;:\;2^{k-1}\le\abs{v}<2^{k}}$; see, \eg, \cite[p.~18]{Triebel_1992}. It is well known that for $0<\alpha<r\in\N$ the Lipschitz norm \eqref{eq_Lip_norm} is equivalent to the norm defined in \eqref{eq_Besov_norm}; see \cite[p.~140]{Triebel_1992}, \cite[p.~144]{Bergh-Loefstroem_1976}.

Some basic properties of these spaces are given in

\begin{theorem}\label{thm_equiv_new}
Let $f\in \Ltwo$, $r\in\N$, $0<\beta<\alpha< r$, and $s\in\N_0$ with $s< \alpha$. Then the following statements are equivalent:

\begin{aufzaehlung}{$(viii)$}
\item $\D f\in \Lip_r(\alpha)$,
\item $\D\int_{\abs{v}\ge \sigma} \bigabs{\wh f(v)}^2\dv =\Oh\pbig{\sigma^{-2\alpha}} \quad (\sigma\to\infty)$,
%\item $\D \int_{\abs{v}\le X} |v|^{2r} \bigabs{\wh{f}(v)}^2\dv=  \Oh\pbig{X^{2r-2\alpha}} \quad (X\to \infty)$,
\item $f^{(s)} \in \Lip_r(\alpha-s)$,
\item $\D\int_{\abs{v}\ge \sigma} |v|^{2s}\bigabs{\wh f(v)}^2\dv =\Oh\pbig{\sigma^{-2(\alpha-s)}} \quad (\sigma\to\infty)$,
%\item $\D \int_{\abs{v}\le X} |v|^{2r+2s} \bigabs{\wh{f}(v)}^2\dv=  \Oh\pbig{X^{2r-2\alpha+2s}} \quad (X\to \infty)$,
\item $\D D^{\{\beta\}}\kern-2pt f \in \Lip_r(\alpha-\beta)$,
\item $\D\int_{\abs{v}\ge \sigma} |v|^{2\beta}\bigabs{\wh f(v)}^2\dv =\Oh\pbig{\sigma^{-2(\alpha-\beta)}} \quad (\sigma\to\infty)$.
%\item $\D \int_{\abs{v}\le X} |v|^{2r+2\beta} \bigabs{\wh{f}(v)}^2\dv=  \Oh\pbig{X^{2r-2\alpha+2\beta}} \quad (X\to \infty)$.
\end{aufzaehlung}
\end{theorem}

\begin{proof}

We first show that \eqref{eq_Besov_norm} is equivalent to
\begin{equation}\label{eq_Besov_seminorm_2}
   \int_{\abs{v}\ge \sigma} |v|^{2\gamma}\bigabs{\wh f(v)}^2\dv =\Oh\pbig{\sigma^{-2(\alpha-\gamma)}} \qquad (\sigma\to\infty)
\end{equation}
for any $\gamma$ such that $0\le \gamma < \alpha$. Indeed, if \eqref{eq_Besov_seminorm_2} holds, then for some constant $c>0$,
\begin{align*}
  &2^{2\gamma(j-1)}\int_{2^{j-1}\le\abs{v}<2^{j}}\bigabs{\wh f(v)}^2\dv\le \int_{2^{j-1}\le\abs{v}<2^{j}}|v|^{2\gamma}\bigabs{\wh f(v)}^2\dv\\[2ex]
  {}\le{}&
  %\sum_{k=j}^\infty \int_{2^{k-1}\le\abs{v}<2^{k}}|v|^{2\gamma}\bigabs{\wh f(v)}^2\dv
  \int_{\abs{v} \ge 2^{j-1}} |v|^{2\gamma}\bigabs{\wh f(v)}^2\dv
  %\le c\sigma^{-2(\alpha-\gamma)}
  \le c2^{-2(\alpha-\gamma)(j-1)}.
\end{align*}
This yields \eqref{eq_Besov_norm}.

For the converse, assume $\sigma\ge 1$ and choose $j\in\N$ such that $2^{j-1}\le \sigma < 2^j$. Then, if \eqref{eq_Besov_norm} holds, one has for a constant $c^\prime>0$,
\begin{align*}
  &\int_{\abs{v} \ge \sigma} |v|^{2\gamma}\bigabs{\wh f(v)}^2\dv\le \sum_{k=j}^\infty \int_{2^{k-1}\le\abs{v}<2^{k}}|v|^{2\gamma}\bigabs{\wh f(v)}^2\dv\\[2ex]
   {}\le{}& \sum_{k=j}^\infty 2^{2\gamma k} \int_{2^{k-1}\le\abs{v}<2^{k}}\bigabs{\wh f(v)}^2\dv
   \le c^\prime  \sum_{k=j}^\infty 2^{-2(\alpha-\gamma) k}=\frac{c^\prime 2^{-2(\alpha-\gamma)j}}{1-2^{-2(\alpha-\gamma)}}%\Oh\pbig{2^{-2(\alpha-\gamma)}}
   =\Oh\pbig{\sigma^{-2(\alpha-\gamma)}}
\end{align*}
as $\sigma\to\infty$, which is \eqref{eq_Besov_seminorm_2}.

It follows from the equivalence of \eqref{eq_Lip_norm}, \eqref{eq_Besov_norm} and \eqref{eq_Besov_seminorm_2} for $\gamma=0$, $\gamma=s$, and $\gamma=\beta$, respectively, that assertions (i), (ii), (iv) and (vi) are equivalent. Furthermore, applying the equivalence of (i) and (ii) to $f^{(s)}$ with Fourier transform $\wh{f^{(s)}}(v)=(iv)^s\wh f(v)$, yields (iii)\,$\Leftrightarrow$\,(iv), and (v)\,$\Leftrightarrow$\,(vi) follow by the same argument, noting Proposition~\ref{prop_Riesz_existence_equiv} and \eqref{eq_FT_Riesz}.
\end{proof}

The above proof building on Besov spaces was designed on the recommendation of one of the referees. For an alternative proof avoiding the theory of Besov spaces see  the remarks after the proof of the following Theorem~\ref{thm_oh}. For a proof of (i)\,$\Leftrightarrow$\,(ii) via the general Butzer-Scherer theorem the reader is referred to \cite{Butzer-Schmeisser-Stens_2013}.

There exists a little-oh analogue of Theorem~\ref{thm_equiv_new}. We state it in a shortened form, using the notation
\[
  \lip_r(\alpha) := \{f\in L^2(\R)\::\: \omega_r(f;\delta; L^2(\R))=o(\delta^{\alpha}), \:\delta\to 0+\}.
\]
\begin{theorem}\label{thm_oh}
Let $f\in L^2(\R)$, $r\in\N$ and $0< \alpha< r$. Then the following statements are equivalent:
\begin{aufzaehlung}{$(ii)$}
\item  $f\in \lip_r(\alpha)$,
\item  $\D\int_{\abs{v}\ge \sigma} \bigabs{\wh f(v)}^2\dv =\oh\pbig{\sigma^{-2\alpha}} \quad (\sigma\to\infty)$,
%
%\item  $\int_{\abs{v}\le X} \abs{v}^{2r} \bigabs{\wh{f}(v)}^2 dv = o(X^{2r-2\alpha}) \qquad (X\to\infty)$.
\end{aufzaehlung}
\end{theorem}

\begin{proof}
First we show that (ii) is equivalent to
\begin{equation}\label{eq_equivalence_to_ii}
  \int_{\abs{v}\le \sigma} \abs{v}^{2r} \bigabs{\wh{f}(v)}^2 dv = o\pbig{\sigma^{2(r-\alpha)})}\qquad (\sigma\to\infty).
\end{equation}

Indeed, assume that (ii) holds, then
\begin{equation}\label{eq_little_oh}
  \int_{\abs{v}\ge\sigma} \bigabs{\wh{f}(v)}^2 dv\,\le\, \frac{c_0(\sigma)}{\sigma^{2\alpha}}\qquad (\sigma>0),
\end{equation}
where $c_0(\cdot)$ is a non-negative function on $(0,\infty)$ such that $c_0(\sigma)\to0$ as $\sigma\to\infty$.
%implies
%%
%\begin{equation}\label{eq_equiv_lip}
%%
%  \int_{\abs{v}\le \sigma} \abs{v}^{2r} \bigabs{\wh{f}(v)}^2 dv = o\pbig{\sigma^{2(r-\alpha)}} \qquad (\sigma\to\infty).
%%
%\end{equation}.
%
Now choose $\kappa \in (0, 1)$ such that $\kappa< 1-\alpha/r$ and split the integral in \eqref{eq_equivalence_to_ii} as follows:
\[
  \int_{\abs{v}\le \sigma} v^{2r} \bigabs{\wh{f}(v)}^2 \dv
  =\int_{\abs{v}\le \sigma^\kappa}v^{2r} \bigabs{\wh{f}(v)}^2 dv + \int_{\sigma^\kappa
  \le \abs{v}\le \sigma}v^{2r} \bigabs{\wh{f}(v)}^2\dv =: \mathcal{I}_1 + \mathcal{I}_2.
\]
Clearly,
\[
  \mathcal{I}_1 \le \sigma^{2r\kappa} \bignorm{\wh{f}\,}_{\Ltwo}^2 = o\pbig{\sigma^{2(r-\alpha)}} \qquad (\sigma\to\infty)
\]
by our choice of $\kappa$.

Concerning $\mathcal{I}_2$, let
\[
    M_k:=\brbig {v\in [\sigma^\kappa, \sigma]\;:\; \sigma2^{-k-1}\le \abs v< \sigma2^{-k}}\qquad (k\in\N_0).
\]
Then \eqref{eq_little_oh} allows us to conclude that for any $\sigma>0$ and $k\in\N_0$ we have
\begin{align*}
\int_{M_k} v^{2r} \bigabs{\wh{f}(v)}^2 \dv {}\le{} (2^{-k}\sigma)^{2r}\int_{M_k} \bigabs{\wh{f}(v)}^2 \dv
  {}\le{} \sup_{\sigma^\kappa \le t \le \sigma} c_0(t)\, 4^{\alpha} (2^{-k}\sigma)^{2(r-\alpha)}.
\end{align*}
Thus
\begin{align*}
\mathcal{I}_2 \le \sum_{k=0}^\infty \int_{M_k} v^{2r} \bigabs{\wh{f}(v)}^2\dv
{}\le{} \sup_{\sigma^\kappa \le t \le \sigma} c_0(t)\, 4^\alpha \sigma^{2(r-\alpha)} \sum_{k=0}^\infty \rez{4^{(r-\alpha)k}}
=  \oh\pbig{\sigma^{2(r-\alpha)}},
\end{align*}
which shows that \eqref{eq_equivalence_to_ii} holds.

Conversely, \eqref{eq_equivalence_to_ii}
%
%Conversely, suppose that (iii) holds. Then there exists a constant $c_1>0$ such that
%%
%\[
% \int_{\abs{v}\le X} v^{2r} \bigabs{\wh{f}(v)}^2 \dv \le c_1\, X^{2r-2\alpha} \qquad (X>0).
%\]
%%
allows us to conclude that for any $\sigma>0$ and $k\in\N_0$ we have
\begin{align*}
\int_{2^k\sigma\le\abs{v}< 2^{k+1}\sigma} \bigabs{\wh{f}(v)}^2 \dv \le {}&
             (2^k\sigma)^{-2r}\int_{2^k\sigma\le\abs{v}< 2^{k+1}\sigma} |v|^{2r}\bigabs{\wh{f}(v)}^2 \dv\\[2ex]
{}\le{}& \sup_{t\ge\sigma}c_0(t)\, 2^{2(r-\alpha)} (2^k\sigma)^{-2\alpha}.
\end{align*}
Thus
\begin{align*}
\int_{\abs{v}\ge \sigma} \bigabs{\wh{f}(v)}^2 dv ={}& \sum_{k=0}^\infty \int_{2^k\sigma\le\abs{v}< 2^{k+1}\sigma}
                        \bigabs{\wh{f}(v)}^2 \dv\\[2ex]
{}\le{}& \sup_{t\ge\sigma}c_0(t) \, \sigma^{-2\alpha} 4^{r-\alpha}\sum_{k=0}^\infty \rez{4^{\alpha k}} = \oh\pbig{\sigma^{-2\alpha}},
\end{align*}
which is (ii).

Now we are ready for the actual proof. Noting that the Fourier transform of the $r$-th order difference $\Delta ^r_h f$ equals
\begin{equation*}
  (\Delta ^r_h f)\wh{\phantom{t}}(v) = \sum_{k=0}^r (-1)^{r-k} \bin{r}{k}e^{ikhv} \widehat f(v)
  =\big(e^{ihv}-1\big)^r\widehat f(v)\qquad (v\in \R),
\end{equation*}
we have by the isometry of the $L^2(\R)$ Fourier transform,
\begin{equation}\label{eq_Delta_norm}
  \bignorm{\Delta_h^rf}_{\Ltwo}^2=
  \bignorm{\big(e^{ihv}-1\big)^r\widehat f(v)}_{\Ltwo}^2 = 2^{2r} \int_\R \Bigabs{\wh{f}(v) \sin^r\Big(\frac{hv}{2}\Big)}^2 \dv\qquad (h>0).
\end{equation}

Since $2\abs{x}/\pi \le \abs{\sin x} \le \abs{x}$ for $\abs{x}\le\pi/2$, we find that for $h>0$,
\begin{align}
      \li(\frac{2}{\pi}\ri)^{2r} \int_{\abs{v}\le\pi/h}\bigabs{\wh{f}(v)}^2& (hv)^{2r} \dv
                                                        \le \bignorm{\Delta_h^rf}_{\Ltwo}^2 \notag\\[2ex]
{}\le{}& \int_{\abs{v}\le\pi/h} \bigabs{\wh{f}(v)}^2 (hv)^{2r} \dv
                                                + 2^{2r} \int_{\abs{v}\ge \pi/h} \bigabs{\wh{f}(v)}^2 \dv.
\label{Delta_bounds}
\end{align}

Suppose that (i) holds. Then the left-hand inequality implies
%Since $2\abs{x}/\pi \le \abs{\sin x} \le \abs{x}$ for $\abs{x}\le\pi/2$, we find that
%%
%\begin{align}
%      \li(\frac{2}{\pi}\ri)^{2r} \int_{\abs{v}\le\pi/h} \bigabs{\wh{f}(v)}^2 (hv)^{2r} \dv
%                                                        \le{} \bignorm{\Delta_h^rf}_{\Ltwo}^2 =\oh\pbig{h^{2\alpha}}\qquad(h\to0+),
%%
%\label{Delta_bounds_new}
%\end{align}
%
\eqref{eq_equivalence_to_ii}, which is equivalent to (ii), as shown above.

%\begin{align}
%      \li(\frac{2}{\pi}\ri)^{2r} \int_{\abs{v}\le\pi/h}\bigabs{\wh{f}(v)}^2& (hv)^{2r} \dv
%                                                        \le \bignorm{\Delta_h^rf}_{\Ltwo}^2 \notag\\[2ex]
%%
%{}\le{}& \int_{\abs{v}\le\pi/h} \bigabs{\wh{f}(v)}^2 (hv)^{2r} \dv
%                                                + 2^{2r} \int_{\abs{v}\ge \pi/h} \bigabs{\wh{f}(v)}^2 \dv.
%\label{Delta_bounds_new_1}
%\end{align}

As to the converse, if (ii) holds, then the equivalence of (ii) and \eqref{eq_equivalence_to_ii} yields that the two integrals on the right-hand side of \eqref{Delta_bounds} are of order $\oh\p{h^{2\alpha}}$ giving (i).
\end{proof}

As to the proof, the authors were inspired by two deep theorems of Titchmarsh \cite[Theorems~84, 85]{Titchmarsh_1948}, \cite{Titchmarsh_1927}, who in turn writes ``The analysis originated with ideas of Bernstein and Szasz on Fourier series''. Titchmarsh had shown among others that for $0<\alpha<1$ the conditions $\int_\R
|f(x+h)-f(x-h)|^2dx=\Oh(|h|^{2\alpha})$ and $\int_{|v|\ge X}|\wh f(v)|^2dv=\Oh(X^{-2\alpha})$, $X\to\infty$, are equivalent. In the foregoing proof we extendend Titchmarsh's ideas to higher order differences in order to get rid of the restriction $\alpha<1$. By an obvious modification of this proof one can give another proof of Theorem~\ref{thm_equiv_new} without using the theory of function spaces.

In order to compare the space $\lip_r(\alpha)$ with $\Rieszalpha$ we will need the following proposition.

\begin{proposition}\label{prop_Riesz_and_oh}
Let $m\in\N_0$ and $\alpha>0$. If $f^{(m)}\in\Rieszalpha$, then
\begin{equation}\label{Riesz1}
   \int_{\abs{v}\ge \sigma} \bigabs{\wh f(v)}^2\dv =\oh\pbig{\sigma^{-2\alpha-2m}} \qquad (\sigma\to\infty).
\end{equation}
\end{proposition}

\begin{proof} By the definition of $\Rieszalpha$ we know that $\int_\R \abs{v}^{2\alpha}\bigabs{v^m\wh{f}(v)}^2 dv$ exists. Hence% given $\eps>0$, there exists a $\sigma_0 >0$ such that
\[
  \sigma^{2\alpha+2m} \int_{\abs{v}\ge \sigma} \bigabs{\wh{f}(v)}^2 \dv
  \le \int_{\abs{v}\ge\sigma} \abs{v}^{2\alpha+2m} \bigabs{\wh{f}(v)}^2\dv =\oh(1)\qquad (\sigma\to \infty),
\]
which shows that \eqref{Riesz1} holds.
\end{proof}

\begin{corollary}\label{cor_Lip_inclusions}
For $r \in\N$ and $0<\beta<\alpha<r$ we have
\begin{equation}\label{eq_Lip_inclusions}
  \Rieszalpha\subsetneqq \lip_r (\alpha) \subsetneqq \Lip_r(\alpha)\subsetneqq \Rieszbeta.
\end{equation}
\end{corollary}

\begin{proof}
The first inclusion follows from Proposition~\ref{prop_Riesz_and_oh} for $m=0$, and the second one from the definition of the Lipschitz spaces involved. The rightmost inclusion is a well-known relation between Besov and Sobolev spaces; see, \eg, \cite{Toft_2004}.
%
% On the other hand, by Theorem~\ref{thm_equiv_new}, $f\in \Lip_r(\alpha)$ implies
%\[
%   \int_{\abs{v}\ge \sigma} |v|^{2\beta}\bigabs{\wh f(v)}^2\dv =\Oh\pbig{\sigma^{-2\alpha+2\beta}} \qquad (\sigma\to\infty),
%\]
%%
%yielding $f\in\Rieszbeta$.

For the fact that equality cannot hold in \eqref{eq_Lip_inclusions}, we refer to Propositions~\ref{prop_counter_allgemein} and \ref{prop_counter_last} in Section~\ref{sec_counterexamples} below.
\end{proof}

\subsection{Wiener amalgam and modulation spaces}\label{sec_Wiener_amalgam_spaces}

For $p, q \in [1,\infty]$ the \emph{Wiener amalgam space} $W(L^p, \ell^q)$ comprises all measurable, locally integrable functions
$f\,:\,\R\to\C$ such that
$$
    \|f\|_{p,q} :=
     \Bigg\{\sum_{n\in\Z}\biggbr{\int_n^{n+1} \abs{f(t)}^p\,dt}^{q/p}\Bigg\}^{1/q} = \Bignorm{\pBig{\norm{f}_{L^p(n,n+1)}}_{n\in\Z}}_{\ell^q}< \infty
$$
with the usual convention applying when $p$ or $q$ is infinite.

The idea of considering $W(L^p, \ell^q)$, as opposed to the space $L^p(\R)= W(L^p, \ell^p)$, is a natural one as it allows one to separate the global from the local behaviour of a function. The idea goes back to N.~Wiener \cite{Wiener_1932} who had considered such special cases  as $W(L^1, \ell^2)$, $W(L^2, \ell^\infty)$. F.~Holland \cite{Holland_1974,Holland_1975} undertook the first systematic study of the general case in 1975. L.~Cooper \cite{Cooper_1960} had met amalgams in his earlier work on positive definite functions. For an excellent, understandable, survey on amalgams on $\R$ and on groups $G$ see Fournier and Stewart%
\footnote{Finbarr Holland but also Jim Stewart are doctoral students of Lionel Cooper, the former at  Cardiff, the latter at Toronto. Hans Feichtinger, who learned to know Stewart as well as John Fournier in Canada in 1986, recalled that Maria Luisa Torres (Athabasca Univ.) received her doctorate under Stewart at McMaster in 1985, the thesis being on amalgams. Thus she is an academic granddaughter of Lionel.} \cite{Fournier-Stewart_1985}; the latter had introduced them in \cite{Stewart_1979}. For a multivariate version see \cite{Feichtinger-Weisz_2007}.

General connections between $W(L^p, \ell^q)$ and $L^p(\R)$ read (see \cite[(2.3), (2.4)]{Fournier-Stewart_1985}),
\begin{align}
  W(L^p, \ell^q)\subset L^p(\R)\cap L^q(\R)\qquad (q\le p),\label{eq_general1}\\[1.2ex]
  W(L^p, \ell^q)\supset L^p(\R)\cap L^q(\R)\qquad (q\ge p).\label{eq_general2}
\end{align}

A first main result in this respect is the dilation invariance of the spaces $W(L^p, \ell^q)$, i.\,e., $f$ belongs to $W(L^p, \ell^q)$ if and only if $f_\lambda:=f(\lambda\cdot)$, where $\lambda\neq 0$, does. For $W(L^p, \ell^p)=L^p(\R)$ this follows by a change of variable $u\to\lambda^{-1}u$, since $\norm{f_\lambda}_{p,p}=\lambda^{-1/p}\norm{f}_{p,p}$. For the spaces $W(L^p, \ell^q)$ with $p\ne q$ this is by no means as simple. The following proposition is a particular case of a very general result; see \cite{Feichtinger_1983,Sugimoto-Tomita_2007,Cordero-Okoudjou_2012}.
\begin{proposition}\label{prop_dilation_invariance}
Let $p,q \in [1,\infty]$ and $\lambda>0$; then for each measurable $f\colon \R\to\C$,
\begin{align}
 2^{-1/q}3^{-1}\lambda^{-1/p-1/q} \norm{f}_{p,q}&\le \norm{f_\lambda}_{p,q} \le 3\lambda^{1-1/p} \norm{f}_{p,q}
            \qquad (\lambda\ge 1),\label{eq_dilation_invariance_1}\\[2ex]
 3^{-1}\lambda^{1-1/p} \norm{f}_{p,q} &\le \norm{f_\lambda}_{p,q} \le 2^{1/q}3 \lambda^{-1/p-1/q} \norm{f}_{p,q}
             \qquad (0< \lambda < 1).\label{eq_dilation_invariance_2}
\end{align}
\end{proposition}

It was pointed out by H.\,G.\ Feichtinger that there is a relationship between $F^2\cap S_1^2$ and $W(L^2, \ell^1)$. In order to make this more precise, we consider the \emph{modulation space} \cite{Feichtinger_1983,Triebel_1983,Groechenig_2001,Feichtinger_2006}%
\footnote{\label{footnote_hans}%
At the workshop \emph{From Abstract to Computational Harmonic Analysis}, held at Strobl (Austria, June 13--19, 2011), conducted by Hans Feichtinger, he informed two of us that in the terminology favoured by his Numerical Harmonic Analysis Group (NuHAG), the space occurring on the right-hand side of \eqref{hans} can be classified as a modulation space and should then be denoted by $M^{2,1}$.}
\begin{equation}\label{hans}
   M^{2,1}=M^{2,1}(\R) := \brBig{f\;:\; f:=\wh{g},\, \, g\in W(L^2, \ell^1)}
\end{equation}
with the norm
\begin{equation}\label{eq_classical_modulation_norm}
    \|f\|_{M^{2,1}} := \sum_{n\in\Z}\biggbr{\int_n^{n+1} \bigabs{\wh{f}(v)}^2\,dv}^{1/2}
   = \Bignorm{\pBig{\bignorm{\wh f\,}_{L^2(n,n+1)}}_{n\in\Z}}_{\ell^1}.
\end{equation}
Thus the elements $f$ of $M^{2,1}$ are exactly the Fourier transforms of the elements $g$ in the amalgam space $W(L^2,\ell^1)$. In fact, Feichtinger was the first to introduce general modular spaces $M^{p,q}(G)$ in Oberwolfach \cite{Feichtinger_1981}.

Since the function $g$ in \eqref{hans} belongs to $W(L^2, \ell^1)\subset \Ltwo\cap \Lone$, it follows that the Fourier transform $\wh g$ can be understood in $\Lone$-sense, which implies that the elements of $M^{2,1}$ can be regarded as continuous $\Ltwo$-functions. Further, since $g(v)=\wh f(-v)$, the modulation space can be equivalently defined as
\begin{equation}\label{eq_hans_2}
   M^{2,1} = \brBig{f\in \Ltwo\cap C(\R) \;:\; \wh f \in W(L^2, \ell^1)}.
\end{equation}

\begin{proposition}\label{propa1}
For each $r\in \N$, $1/2<\alpha<r$ and $h>0$ there holds the inclusion chain
\begin{equation}\label{eq_Walpha_inclusion}
 \Wrtwo \cap C(\R)\subsetneqq \Rieszalpha\cap C(\R) \subsetneqq\ M^{2,1} \subsetneqq F^2\cap S_h^2.
\end{equation}
\end{proposition}

\begin{proof}
First we note that (cf.~Corollary~\ref{cor_Lip_inclusions} and Proposition~\ref{prop_counter_last}\,a)
\begin{equation}\label{eq_Walpha_Wbeta_inclusion}
\Rieszbeta\cap C(\R)\subsetneqq \Rieszalpha\cap C(\R) \qquad (\alpha < \beta).
\end{equation}
The leftmost inclusion now follows by \eqref{eq_W_r}. The second inclusion can be found in \cite[(0.3), Thm.~3.1, (3.2)]{Toft_2004}, \cite[Thm~2.14]{Wang-Hudzik_2007}.
%
%Next assume $f\in \Rieszalpha\cap C(\R)$. Noting \eqref{eq_hans_2}, we have only to show that $\wh f \in W(L^2,\ell^1)$. Now, $|v|^\alpha \wh f (v) =\phi(v)$ for some $\phi\in\Ltwo$, and it follows that $g(v):=|v|^{-\alpha}\phi(-v)=\wh f(-v)$ belongs to $\Ltwo\cap \Lone$, noting that $\alpha>1/2$. Since $f$ is continuous, we have $f(x)=\wh g(x)$ for all $x\in\R$. It remains to show that $g\in W(L^2,\ell^1)$. In this respect we have
%%
%\begin{align*}
%\sum_{n=1}^\infty \biggbr{\int_n^{n+1} \bigabs{\wh f(v)}^2\,dv}^{1/2}
%{}={} & \sum_{n=1}^\infty \biggbr{\int_n^{n+1} \frac{|\phi(v)|^2}{|v|^{2\alpha}}\,dv}^{1/2}\\[2ex]
%%
%{}\le{}& \sum_{n=1}^\infty \rez{n^\alpha}\biggbr{\int_n^{n+1} \abs{\phi(v)}^2\,dv}^{1/2} \\[2ex]
%%
%{}\le{}&\Biggbr{\sum_{n=1}^\infty \rez{n^{2\alpha}}}^{1/2}\cdot\Biggbr{\sum_{n=1}^\infty \int_n^{n+1} \abs{\phi(v)}^2\,dv}^{1/2}\\[2ex]
%%
%{}\le{}& \Biggbr{\sum_{n=1}^\infty \rez{n^{2\alpha}}}^{1/2}\cdot \|\phi\|_{L^2(\R)}<\un.
%\end{align*}
%%
%The same bound holds for the sum $\sum_{n=-\infty}^{-2}$, allowing us to conclude that $\wh f\in W(L^2,\ell^1)$ and hence $f\in M^{2,1}$.  This proves the second inclusion of \eqref{eq_Walpha_inclusion}
Equality cannot hold since $\Rieszalpha\cap C(\R) =\ M^{2,1}$ for some $\alpha >1/2$ would imply
\[
  \Rieszalpha=\Rieszgamma=M^{2,1}\qquad (1/2<\gamma < \alpha),
\]
which contradicts \eqref{eq_Walpha_Wbeta_inclusion}.

Concerning the third inclusion of \eqref{eq_Walpha_inclusion}, if $f\in M^{2,1}$, then $f=\wh{g}$ with $g\in W(L^2,\ell^1)$. Using the standard inclusion relation  for Wiener amalgam spaces \eqref{eq_general1}, we deduce $g\in L^1(\R)\cap L^2(\R)$, and it follows by \cite[Prop.~5.1.2, Prop.~5.2.1]{Butzer-Nessel_1971} that $f=\wh{g}\in L^2(\R)\cap C(\R)$. Furthermore $\wh f(v) = \wh{\wh g}(v)=g(-v)\in L^1(\R)$. Altogether we see that $f$ belongs to $F^2$.

Next we show that $M^{2,1}$ is a subset of $S_1^2$ as well. By a fundamental result on Fourier transforms in Wiener
amalgam spaces \cite[Theorem~2]{Holland_1974}, \cite[Theorem~8]{Holland_1975}, \cite[Theorem~2.8]{Fournier-Stewart_1985}, it follows
that $g\in W(L^2,\ell^1)$ implies $\wh{g}\in W(L^\infty, \ell^2)$. Thus $M^{2,1}$ is a subset of $W(L^\infty, \ell^2)$, which is obviously a subset of $S_1^2$. Now let $h>0$. Since $f_h(\cdot):=f(h\,\cdot\,)$ belongs to $M^{2,1}$ by the dilatation invariance (Proposition~\ref{prop_dilation_invariance}), it belongs to $S^2_1$ as well, and so $f\in S^2_h$.
A counterexample, showing that the rightmost inclusion in \eqref{eq_Walpha_inclusion} is strict, is provided in Proposition~\ref{prop_counter_allgemein}\,(iv); see also \cite{Butzer-Schmeisser-Stens_2014a}.
\end{proof}
\subsection{The readapted modulation space $\Mn$}
Let $f\in M^{2,1}$. The dilation invariance of the Wiener amalgam space implies that
\begin{equation}\label{M_new1}
\sum_{n\in\Z} \sm{}
\end{equation}
is finite for all $h>0$. However, if $f\not\equiv 0$, then, as a function of $h$, the expression \eqref{M_new1} is not bounded. Indeed, either the term for $n=0$ or that for $n=-1$ or both approach $+\infty$ as $h \to 0+$. If we omit $n=0$ and $n=-1$ in the summation, then this modified expression \zit{M_new1} may remain bounded. For example, for a bandlimited function $f$ it becomes zero when $h$ is sufficiently small. Thus the size of that modified expression may indicate the deviation from bandlimitedness. This is the motivation for specifying a subspace $\Mn$ of $M^{2,1}$ as follows.

The space $\Mn$ comprises all measurable $f\colon \R\to \C$ such that
\begin{equation}\label{eq_def_norm_Mn}
  \norm{f}_{\Mn} := \norm{f}_{\Ltwo} + \mathcal{N}(f)<\infty
\end{equation}
where
\begin{equation}\label{M_new2}
\mathcal{N}(f) := \sup_{0<h\le 1} \sum_{n\in\Z\setminus\{-1, 0\}}\sm{}.
\end{equation}

Obviously $\|\cdot\|_{\Mn}$ defines a norm on $\Mn$ with
\begin{equation}\label{eq_normvergleich}
  \norm{f}_{\Ltwo}\le \norm{f}_{M^{2,1}}\le 2\norm{f}_{\Mn}\qquad(f\in \Mn),
\end{equation}
\ie, $\Mn\subset M^{2,1}\subset \Ltwo$. Further,  $\Mn\subsetneqq M^{2,1}$ as is shown by a counterexample in Proposition~\ref{prop_counterex2}.

\begin{theorem}\label{thm_completeness}
$\Mn$ is a Banach space.
\end{theorem}

In regard to the proof of this basic theorem we need the following preliminary lemma; see e.\,g.~\cite[pp.~116--117]{Royden_1968}.

\begin{lemma}\label{la_Royden}
  A normed linear space is complete if and only if every absolute convergent series is convergent, i.\,e.,
  \[
    \sum_{k=1}^\un\norm{f_k}_X<\un \implies \sum_{k=1}^\un f_k \text{\ \ is convergent in $X$.}
  \]
\end{lemma}

{\renewcommand{\proofname}{Proof of Theorem~\ref{thm_completeness}}
\begin{proof}
We set for $h\in (0,1]$,
\begin{equation}\label{Nhf}
  \mathcal{N}_h(g):=\sum_{n\in\Z\setminus\{-1,0\}}\brbigg{\frac1h\int_{n/h}^{(n+1)/h} \bigabs{\wh g(v)}^2\dv}^{1/2}.
\end{equation}
Obviously, $\mathcal{N}_h$ defines a seminorm on $\Mtwo$ and also on $\Mn$. In view of Proposition~\ref{prop_dilation_invariance} there holds
\[
   \Nh(g)\le C(h)\norm{g}_{\Mtwo}\qquad (g\in\Mtwo)
\]
with a constant $C(h)$ depending on $h\in(0,1]$.

Now, let $\sum_{k=1}^\un f_k$ be a series in $\Mn$ with
\begin{equation}\label{1}
    \sum_{k=1}^\un\norm{f_k}_{\Mn}<\un.
\end{equation}
It follows by \eqref{eq_normvergleich} that also $\sum_{k=1}^\un\norm{f_k}_{\Mtwo}<\un$, and then by Lemma~\ref{la_Royden}, noting the completeness of $\Mtwo$, that there exists a function $f\in \Mtwo$ such that
\begin{equation}\label{2}
  \lim_{m\to\un}\norm{f- S_m}_{\Mtwo}=0,
\end{equation}
where $S_m:=\sum_{k=1}^m f_k$.

We are going to show that $f\in\Mn$ and $\lim_{m\to\un}\norm{f-S_m}_{\Mn}=0$.  In this respect we have for each $h\in (0,1]$,
\begin{align*}
  \Nh(f)\le {}\Nh\p{f- S_m}+\Nh\p{S_m}%\\[1ex]
  {}\le{} & C(h)\norm{f- S_m}_{\Mtwo} + \sum_{k=1}^m \Nh(f_k)\\[1ex]
  {}\le{}& C(h)\norm{f- S_m}_{\Mtwo} + \sum_{k=1}^m \CN(f_k).
\end{align*}
Letting $m\to\infty$
%by \eqref{2},
%\[
%  \Nh(f)\le \sum_{k=1}^\un \CN(f_k)\le \sum_{k=1}^\un\norm{f_k}_{\Mn}.
%\]
%%
and taking the supremum over $h\in (0,1]$ yields by \eqref{2} and \eqref{1},
\[
  \CN(f)\le \sum_{k=1}^\un\norm{f_k}_{\Mn}<\un.
\]
Since we already know that $f\in M^{2,1}\subset \Ltwo$ it follows that $f\in \Mn$.

Similarly, we have for $N>m$,
\begin{align*}
  \Nh\p{f- S_m}\le{} & \Nh\p{f- S_N}+\Nh\p{S_N-S_m}\\[1ex]
%  {}\le{}& C(h)\norm{f- S_N}_{\Mtwo} + \Nh\pbigg{\sum_{k=m+1}^N f_k}\\[1ex]
%  {}\le{}& C(h)\norm{f- S_N}_{\Mtwo} + \sum_{k=m+1}^N \Nh(f_k)\\[1ex]
   {}\le{}& C(h)\norm{f- S_N}_{\Mtwo} + \sum_{k=m+1}^N \CN(f_k),
\end{align*}
and we obtain by the same arguments as above that
\[
  \CN\p{f- S_m}\le \sum_{k=m+1}^\un \CN(f_k)\le \sum_{k=m+1}^\un\norm{f_k}_{\Mn}.
\]
By \eqref{1}, the latter series becomes arbitrarily small for $m$ large enough, and hence, noting \eqref{2} and \eqref{eq_normvergleich}, we obtain as desired,
\[
  \lim_{m\to\un} \norm{f- S_m}_{\Mn}= \lim_{m\to\un}\brbig{\norm{f- S_m}_{\Ltwo}+\CN\p{f- S_m}}=0.
\]

Altogether, we have shown that the series  $\sum_{k=1}^\un f_k$ converges in $\Mn$, and the completeness of $\Mn$ now follows by Lemma~\ref{la_Royden}.
\end{proof}}

As an extension of Proposition~\ref{propa1} there holds

\begin{proposition}\label{propa1_star}
For each  $r\in \N$, $1/2<\alpha<r$ and $h>0$ there holds the inclusion chain
\begin{gather}\label{sobolev2_star}
W^{r,2}(\R)\cap C(\R) \subsetneqq
  \Rieszalpha\cap C(\R) \subsetneqq\Mn\subsetneqq M^{2,1} \subsetneqq F^2\cap S_h^2.
\end{gather}
\end{proposition}

\begin{proof}
It remains to show the second inclusion. If $f\in\Rieszalpha\cap C(\R)$, then $f\in M^{2,1}$ by Proposition~\ref{propa1}, and hence $\norm{f}_{\Ltwo}<\un$. In order to estimate $\mathcal N(f)$, we proceed in a similar way as in the corresponding proof of Proposition~\ref{propa1}. Indeed, with $\phi$ as in that proof one has for $0<\eta\le 1$
\begin{align*}
\sum_{n=1}^\infty \brbigg{\frac1\eta\int_{n/\eta}^{(n+1)/\eta} \bigabs{\wh f(v)}^2\,dv}^{1/2}\!\!\!
{}\le{} &  \brbigg{\sum_{n=1}^\infty \frac{\eta^{2\alpha-1}}{n^{2\alpha}}}^{1/2}\!\!\!\cdot\!
\brbigg{\sum_{n=1}^\infty \int_{n/\eta}^{(n+1)/\eta} \abs{\phi(v)}^2\,dv}^{1/2}
\\[2ex]
%
%{}\le{} & h \brbigg{\sum_{n=N}^\infty \rez{n^2}}^{1/2}\!\!\!\cdot\!
%\li\{\sum_{n=N}^\infty\frac1h \int_{n/h}^{(n+1)/h} \abs{\phi(v)}^2\,dv\ri\}^{1/2}
%\\[2ex]
{}\le{}& \eta^{\alpha-1/2} \brbigg{\sum_{n=1}^\infty{n^{-2\alpha}}}^{1/2}\cdot\|\phi\|_{L^2(\R)}\,.
\end{align*}
The same holds for the sum $\sum_{n=-\infty}^{-1}$, showing that $\mathcal N(f)<\un$. It follows that $\Rieszalpha\cap C(\R) \subset\Mn$. Equality cannot hold in view of Proposition~\ref{prop_counter_last}\,b), noting that $\Rieszalpha\cap C(\R) \subset \Lip_r(\alpha)\cap C(\R)$ by Corollary~\ref{cor_Lip_inclusions}.
\end{proof}

\section{Norms, distances, derivatives and rates of convergence}\label{distance}
\subsection{Distances from the Bernstein space $B^2_\sigma$.}
For $1\le q \le 2$, the distance of a function $f\in F^2$ from the Bernstein space $B_\sigma^2$ is defined by
\[
  \dist_q(f, B_\sigma^2):=\inf_{g\in B^2_\sigma}\bignorm{\wh f-\wh g}_{L^q(\R)}.
\]
Note that $f\in F^2$ implies $\wh f \in \Lone\cap\Ltwo$ and hence $\wh f \in \Lq$ for $1\le q\le 2$. For this distance we have (see \cite{Butzer-Schmeisser-Stens_2013}),
\begin{proposition}\label{prop_dist_representation}
a) If $f\in F^2$, then for each $q\in [1, 2]$,
\[
 \dist_q(f, B_\sigma^2)= \brbigg{\int_{\abs{v}\ge\sigma} |\wh f(v)|^q\dv}^{1/q}.
\]
b) Let $f\in \Wstwo\cap\Fstwo$ for some $s\in \N$. Then for $q\in[1,2]$ and $0\le k\le s$,
\[
 \dist_q(f^{(k)}, B_\sigma^2)= \brbigg{\int_{\abs{v}\ge\sigma} |v^k\wh f(v)|^q\dv}^{1/q}.
\]
\end{proposition}

Proposition~\ref{prop_dist_representation} shows in particular that the distances $\dist_q(f,B_\sigma^2)$ and $\dist_q(f^{(k)}, B_\sigma^2)$ tend to zero for  $\sigma\to \infty$.

\subsection{Derivative-free estimates for the distance from $B^2_\sigma$}
Next we state some derivative-free estimates for the distances $\dist_q(f, B_\sigma^2)$; see \cite{Butzer-Schmeisser-Stens_2013}.
\begin{proposition}\label{cor_special_cases} Assume $q\in[1,2]$.

\smallskip\noindent
a) Let $f\in F^{2}$ and $r\in\N$. Then,
\begin{align*}
 \dist_q(f, B_\sigma^2) %=  \bigg\{\int_{\abs{v}\ge\sigma}
%   \bigabs{\wh{f}(v)}^q\,dv\bigg\}^{1/q}\\[2ex]
%
   \le c_{r,q}\bigg\{\int_\sigma^\infty v^{-q/2}\big[\omega_r(f,v^{-1},L^2(\R))\big]^q dv\bigg\}^{1/q}.
\end{align*}
If $f\in \Lip_r(\alpha)$, $1/q-1/2 < \alpha \le r$, then
\[
   \dist_q(f, B_\sigma^2)  =   \Oh\big(\sigma^{-\alpha -1/2 +1/q}\big) \qquad (\sigma\to\infty).
\]

\noindent
b) If $f\in \Wstwo\cap \Fstwo$ for some $s\in\N_0$, then for each $0\le k \le s$ and $r\in\N$,
\begin{align*}
  \dist_q(f^{(k)},B^2_\sigma) %= \bigg\{\int_{|v|>\sigma}\big|v\wh f(v) \big|^q\bigg\}^{1/q}\\[2ex]
  {}\le{}&  c_{s,r,q}\bigg\{\int_\sigma^\infty\!\!\! v^{-q/2}\big[\omega_r(f^{(s)},v^{-1},L^2(\R))\big]^q dv\bigg\}^{1/q}\\[2ex]
{}={}& \Oh\big(\sigma^{-\alpha-s+k-1/2+1/q}\big)\qquad (\sigma\to\un),
\end{align*}
the latter holding provided $f^{(k)}\in \Lip_r(\alpha)$, $1/q-1/2 < \alpha \le r$.
\end{proposition}

It should be noted that the integrals in a) and b) may be infinite although the distances on the left-hand sides are finite. If, however, $f$ satisfies a Lipschitz condition of a certain order, then the integrals are finite, as seen above.

\subsection{Distance of $\Mn$-functions from $B^2_\sigma$}
Next we consider distance estimates in the setting of the space $\Mn$, introduced in Section~\ref{sec_Wiener_amalgam_spaces}.
\begin{proposition}\label{dist_M_star}
a) If $f\in \Mn$, then
\[
   \dist_q(f, B_\sigma^2) =
\begin{cases}
         \Oh(1), & \text{ for } q=1, \\[1.2ex]
         \Oh\big(\sigma^{-1+1/q}\big), & \text{ for } q\in (1, 2]
\end{cases}  \qquad (\sigma\to\infty).
\]

\noindent b) Let $r\in\N$ and $q\in [1,2]$. If $f\in W^{r,2}(\R)\cap C(\R)$ and $f^{(r)}\in \Mn$, then

\begin{equation}\label{dist_M_star1}
    \dist_q(f, B_\sigma^2) = \Oh\big(\sigma^{-r-1+1/q}\big) \qquad (\sigma \to\infty),
\end{equation}
and for $1\le s\le r$,
\begin{equation}\label{dist_M_star2}
  \dist_q(f^{(s)}, B_\sigma^2) =
  \begin{cases}
    \Oh(1), & \text{if } s=r=q=1, \\[1.2ex]
    \Oh\big(\sigma^{-r-1+s+1/q}\big), &\text{otherwise}
  \end{cases}\qquad (\sigma\to\infty).
\end{equation}
\end{proposition}

\begin{proof}
What we want is an estimate of
\begin{equation}\label{eq_wanted_estimate}
\dist_q(f, B_\sigma^2) = \brbigg{\int_{\abs{v}\ge\sigma} \bigabs{\wh{f}(v)}^q\,dv}^{1/q}
= \brBigg{\sum_{n\in\Z\setminus\br{-1,0}} \int_{n/h}^{(n+1)/h} \bigabs{\wh{f}(v)}^q dv}^{1/q}
\end{equation}
for $\sigma \ge 1$ and $h=\sigma^{-1}$.

Indeed, using Hölder's inequality with the exponents $2/(2-q)$ and $2/q$, and then the inequality $\bkbig{\sum \abs{a_n}^\alpha}^{1/\alpha} \le \bkbig{\sum \abs{a_n}^\beta}^{1/\beta}$, $0<\beta\le \alpha< \infty$, for $\alpha=q/2$ and $\beta=1/2$, we find for $q\in [1, 2)$ and $0< h\le 1$,
\begin{align}
&\brBigg{\sum_{n\in\Z\setminus\br{-1,0}} \int_{n/h}^{(n+1)/h} \bigabs{\wh{f}(v)}^q\dv}^{1/q}\notag\\[2ex]
{}\le{}& \brBigg{\sum_{n\in\Z\setminus\br{-1,0}} \rez{h^{1-q/2}}\brbigg{\int_{n/h}^{(n+1)/h} \bigabs{\wh{f}(v)}^2 dv}^{q/2}}^{1/q}\notag\\[2ex]
{}={}&
\brBigg{\sum_{n\in\Z\setminus\br{-1,0}} \brbigg{\rez{h^{2/q-1}}\int_{n/h}^{(n+1)/h} \bigabs{\wh{f}(v)}^2 \dv}^{q/2}}^{1/q}\notag\\[2ex]
{}\le{}&
\sum_{n\in\Z\setminus\br{-1,0}} \brbigg{\rez{h^{2/q-1}}\int_{n/h}^{(n+1)/h} \bigabs{\wh{f}(v)}^2 \dv}^{1/2}\notag\\[2ex]
{}={}&
h^{1-1/q} \,\sum_{n\in\Z\setminus\br{-1,0}} \brbigg{\rez{h} \int_{n/h}^{(n+1)/h} \bigabs{\wh{f}(v)}^2 \dv}^{1/2}. \label{Ungl_Kette}
\end{align}
Trivially, this estimate extends to $q=2$, and so it holds for all $q\in[1,2]$.

Under the hypothesis $f\in\Mn$ we have that $\mathcal N(f)<\infty$ (cf.~\eqref{M_new2}), and we obtain
\[
  \brBigg{\sum_{n\in\Z\setminus\br{-1,0}} \int_{n/h}^{(n+1)/h} \bigabs{\wh{f}(v)}^q \dv}^{1/q}
   \le  h^{1-1/q} \mathcal N(f)\qquad (0< h\le 1).
\]
Now, if $\sigma \ge 1$, then $h$ may be chosen as $\sigma^{-1}$, and doing so, we obtain from \eqref{eq_wanted_estimate},
\[
     \dist_q(f, B_\sigma^2) \le  \mathcal N(f)\sigma^{-1+1/q}\qquad (\sigma\ge 1),
\]
which is the conclusion of statement a).

Under the hypotheses of statement b), we have
\[
  \wh{f^{(r)}}(v) =  (iv)^r \wh{f}(v) \quad \hbox{ a.\,e.};
\]
see \cite[p.~214, Proposition~5.2.19]{Butzer-Nessel_1971}. Hence $f^{(r)}\in \Mn$ implies
\begin{equation}\label{diff_scaled_mod}
  \mathcal N\pbig{f^{(r)}}= \sup_{0<h\le 1} \sum_{n\in\Z\setminus\br{-1,0}}  \brbigg{\rez{h}\int_{n/h}^{(n+1)/h} \bigabs{v^r \wh{f}(v)}^2\,dv}^{1/2}<\infty.
\end{equation}

Now we note that for $n\in \Z\setminus\br{-1,0}$, we have
\[
\int_{n/h}^{(n+1)/h} \bigabs{\wh{f}(v)}^q\dv \le h ^{rq} \int_{n/h}^{(n+1)/h} \bigabs{v^r \wh{f}(v)}^q\dv\,.
\]
Employing the estimate \zit{Ungl_Kette} with the role of $\wh{f}(v)$ taken by $v^r \wh{f}(v)$, we obtain
\begin{multline*}
  \brBigg{\sum_{n\in\Z\setminus\br{-1,0}} \int_{n/h}^{(n+1)/h} \bigabs{\wh{f}(v)}^q \dv}^{1/q}\\
 {}\le h^{r+1-1/q} \sum_{n\in\Z\setminus\br{-1,0}} \brbigg{\rez{h} \int_{n/h}^{(n+1)/h} \bigabs{v^r \wh{f}(v)}^2\dv}^{1/2}.
\end{multline*}
The series on the right-hand can be estimated by \eqref{diff_scaled_mod}, and we obtain \eqref{dist_M_star1} by choosing $h=\sigma^{-1}$ as in the first part of the proof.

Next we note that \zit{diff_scaled_mod} implies that $v^s \wh{f}(v) \in L^1(\R)$ for $s=1, \dots, r$. Thus the distance $\dist_q(f^{(s)}, B_\sigma^2)$ is well defined. Furthermore we observe that, under the hypotheses of \eqref{dist_M_star2}, the derivative $f^{(s)}$ satisfies the hypothesis of statement a) when $r=1$, and the hypotheses of \eqref{dist_M_star1} with $r$ replaced by $r-s$, when $r\ge 2$. Hence \eqref{dist_M_star2} follows from the preceding results.
\end{proof}

Our results for the distance $\dist_2(f, B_\sigma^2)$ in conjunction with Theorem~\ref{thm_equiv_new} imply the following  inclusions. This theorem covers the most important theoretical results of our paper.

\begin{theorem}\label{M_star_compare2_new}
For $r \in\N$ and $1/2<\alpha<r$ we have
\begin{gather}
\Wrtwo \cap C(\R)\subsetneqq\big\{f\in L^2(\R)\cap C(\R)\;:\; D^{\{\alpha\}} f\in L^2(\R)\big\}\label{eq_inclusion_Sobo-Riesz}\\[2ex]
{}={} \Rieszalpha\cap C(\R)\subsetneqq\Lip_r(\alpha) \cap C(\R){}\subsetneqq{}  \Rieszhal \cap \Mn\label{eq_set_v_alphax}\\[2ex]
%{}={}& \big\{f\in L^2(\R)\cap  C(\R)\;:\; |v|^\alpha\widehat f(v)\in L^2(\R)\big\}\\[2ex]
%
{}\subsetneqq{} \Mn\subsetneqq \pbig{\Rieszhal\cap F^2} \cup \Mn
{}\subsetneqq{} \Lip_r\big(\hal) \cap F^2.\label{eq_other_inclusionsx}
\end{gather}
\end{theorem}

\begin{proof}
The equality between \eqref{eq_inclusion_Sobo-Riesz} and \eqref{eq_set_v_alphax} is contained in Proposition~\ref{prop_Riesz_existence_equiv}. Hence the strict inclusion in \eqref{eq_inclusion_Sobo-Riesz} follows by Proposition~\ref{propa1}, and the first inclusion in \eqref{eq_set_v_alphax} was proved in Corollary~\ref{cor_Lip_inclusions}.

As to the second inclusion in \eqref{eq_set_v_alphax}, choose $\beta$ with $\frac12 < \beta < \alpha$. Then by \eqref{eq_Lip_inclusions} and \eqref{sobolev2_star},
\[
  \Lip_r(\alpha)\cap C(\R)\subset \Rieszbeta\cap C(\R) \subset\Mn.
\]
On the other hand, we have again by \eqref{eq_Lip_inclusions} that $\Lip_r(\alpha)\subset \Lip_r(\hal)\subset \Rieszhal$. This yields the second inclusion in  \eqref{eq_set_v_alphax}, and the next two inclusions are obvious.

Since $\Rieszhal\subset \Lip_r(\hal)$ by Corollary~\ref{cor_Lip_inclusions}, it remains to show that  $\Mn\subset\Lip_r(\hal)\cap F^2$. Here we use Proposition~\ref{dist_M_star}, showing $f\in \Mn$ implies $\dist_2(f,B^2_\sigma)=\Oh(\sigma^{-1/2})$, and conclude by Theorem~\ref{thm_equiv_new}\,(i)$\,\Leftrightarrow\,$(ii) that this $\Oh$-condition is equivalent to $f\in \Lip_r(\hal)$. Note that assertion (ii) of Theorem~\ref{thm_equiv_new} can be expressed in terms of the distance functional as $\dist_2(f,B^2_\sigma)=\Oh(\sigma^{-\alpha})$. Further, we have $\Mn\subset F^2$ in view of \eqref{sobolev2_star}.

For the fact that the inclusions in \eqref{eq_set_v_alphax} and \eqref{eq_other_inclusionsx} are strict, see Proposition~\ref{prop_counter_allgemein}\,(i), noting that $\lip_r(\alpha)\subset\Lip_r(\alpha)$, as well as Propositions~\ref{prop_counter_last}\,b), \ref{prop_counterex3}, \ref{prop_counterex4}, \ref{prop_counterex5}.
\end{proof}

It follows from Theorem~\ref{M_star_compare2_new} that $f\in \Mn$ implies $f\in F^2$. More generally, $f^{(s)}\in \Mn$ implies $f\in F^{s,2}$. The same holds with $\Mn$ replaced by $\Lip_r(\alpha)\cap C(\R)$ for any $\alpha>1/2$.

Also by Theorem~\ref{M_star_compare2_new} we have that $f\in \Mn$ implies in particular $f\in\Lip_r(\frac12)$ and hence $\omega_r\big(f;\delta;L^2(\R)\big)=\Oh(\delta^{1/2})$ for $\delta\to 0+$. On the other hand, $\Mn$ is a \emph{proper} subspace of $\Lip_r(\frac12)$, and one may ask, whether the order $\Oh(\delta^{1/2})$ can be improved for $\Mn$-functions, \ie, can $\Oh(\delta^{1/2})$ be replaced by $\oh(\delta^{1/2})$. The answer is \emph{no} as will be shown in the following proposition. %The space $\Lip(\alpha)$ for $\alpha> 1/2$ delivers better approximation orders.

\begin{proposition}\label{prop_oh}
There exists a function $f\in \Mn$ such that $\omega_r(f;\delta; L^2(\R)) \ne \oh\pbig{\delta^{1/2}}$ as $\delta \to0+$, and also $\dist_2(f,B^2_\sigma)\neq \oh\pbig{\sigma^{-1/2}}$ as $\sigma\to\infty$.
\end{proposition}
\begin{proof} See Section~\ref{sec_counterexamples}, after the proof of Proposition~\ref{prop_counterex3}.\end{proof}

\section{The extension of basic relations from  $B_\sigma^2$ to larger spaces}\label{sec_extensions}

\subsection{The classical sampling formula }
The extension of the classical Whittaker-Kotel'nikov-Shannon sampling theorem (the WKS, or the CSF in the terminology of \cite{Butzer-Ferreira_et_al_2011b}) to non-bandlimited functions, known for some time now (see, e.\,g., \cite{Brown_1967,Brown_1968,Weiss_1963}, \cite[Section 11.3]{Higgins_1996}, \cite[Sections 3.5, 3.8]{Zayed_1993}), reads:

\begin{theorem}\label{ASF}
Let $f\in F^2\cap S_h^2$ with $h>0$. Then
\begin{equation}\label{ASF1}
   f(t) = \sum_{k\in\Z} f(hk)\,\sinc(h^{-1}t-k) + (\Rhf{WKS})(t)\qquad (t\in\R),
\end{equation}
where
\begin{equation}\label{ASF2}
  (\Rhf{WKS})(t) = \rez{\sqrt{2\pi}}\sum_{k\in\Z}\li(1-e^{-i 2\pi kt/h}\ri) \int_{(2k-1)\pi/h}^{(2k+1)\pi/h} \wh{f}(v)\,e^{i vt}\,dv\,.
\end{equation}
Furthermore,
\begin{equation}\label{ASF3}
  \bigabs{(\Rhf{WKS})(t)}\le\sqrt{\frac{2}{\pi}}\int_{\abs{v}\ge\pi/h} \bigabs{\wh{f}(v)}\,dv = \sqrt{\frac{2}{\pi}}\, \dist_1(f, B_{\pi/h}^2).
\end{equation}
The sampling series converges absolutely and uniformly on $\R$.
\end{theorem}

Let us note that the error bound in \eqref{ASF3} is sharp in the sense that there exists an \emph{extremal} function $f$ and a point $t\in\R$ for which equality holds in \eqref{ASF3}. Such an extremal is, \eg, given by $\sinc(2h^{-1}t-1)$. See \cite{Brown_1967}, \cite[p.~92]{Zayed_1993}, and for a more detailed discussion of extremals \cite[pp.~119--120]{Higgins_1996}.

If $f\in B^2_\sigma$ for some $\sigma\le \pi/h$, then $\dist_1(f, B_{\pi/h}^2)=0$ and so WKS is an immediate corollary.

With the help of \eqref{ASF3} and Proposition~\ref{dist_M_star}\,b), we obtain the following rate of convergence for the remainder.

\begin{corollary}\label{ASF_Sob}
If $f\in W^{m,2}(\R)\cap C(\R)$ and $f^{(m)}\in \Mn$ for some $m\in\N$, then
\[
     \absbig{(\Rhf{WKS})(t)}=\Oh(h^{m})\qquad (h\to 0+).
\]
\end{corollary}

% -------------------------------------------------------------------------------------------------------------------

\subsection{Bernstein's inequality}\label{bernst_ineq}

The aim of this section is to generalize the well-known Bernstein inequality for $f\in B^2_\sigma$, namely,
\begin{equation}\label{eq_Bernstein}
    \norm{f^{(s)}}_{L^{2}(\R)}\le \sigma^s \norm{f}_{L^{2}(\R)} \qquad (s\in\N;\sigma>0)
\end{equation}
beyond bandlimited functions.
\begin{theorem}\label{bernst_L2}
Let $f\in W^{s,2}(\R)\cap F^{s,2}$ for some $s\in\N$. Then, for any $\sigma>0$,
\[
\|f^{(s)}\|_{L^2(\R)}\le\sigma^s \|f\|_{L^2(\R)} +  \dist_2(f^{(s)}, B_\sigma^2).
\]
\end{theorem}

\begin{proof}
Define
\[
 f_1(t) := \rez{\sqrt{2\pi}} \int_{\abs{v}\ge\sigma} \wh{f}(v) e^{itv}\dv
\]
and $f_0:=f-f_1$. Then $f_0\in B_\sigma^2,$ and so by Bernstein's inequality \eqref{eq_Bernstein}
\begin{equation}\label{bernst4}
   \|f_0^{(s)}\|_{L^2(\R)} \le \sigma^s \|f_0\|_{L^2(\R)}\,.
\end{equation}
Since
\[
  f_1^{(s)}(t) = \frac{1}{\sqrt{2\pi}}\int_{\abs{v}\ge \sigma} (iv)^s\wh f(v) e^{ivt}\dt
\]
by \eqref{eq_derivative_representation}, then we have by the isometry of the Fourier transform in $L^2(\R)$ and Proposition~\ref{prop_dist_representation}\,b),
\begin{equation}\label{bernst5}
  \|f_1^{(s)}\|_{L^2(\R)} =  \biggbr{\int_{\abs{v}\ge\sigma}\bigabs{v^s\wh{f}(v)}^2 \dv}^{1/2} = \dist_2(f^{(s)}, B_\sigma^2)\,.
\end{equation}
Again by the isometry
\begin{align*}
\bignorm{f^{(s)}}_{L^2(\R)}^2 {}={}& \bignorm{(iv)^s\wh{f}(v)}_{L^2(\R)}^2 = \int_{\abs{v}\le\sigma} \bigabs{v^s\wh{f}(v)}^2\dv +
                                                                            \int_{\abs{v}\ge\sigma} \bigabs{v^s\wh{f}(v)}^2\dv \\[2ex]
{}={}&\Bignorm{\wh{f_0^{(s)}}}_{L^2(\R)}^2 + \Bignorm{\wh{f_1^{(s)}}}_{L^2(\R)}^2
=
\bignorm{f_0^{(s)}}_{L^2(\R)}^2 + \bignorm{f_1^{(s)}}_{L^2(\R)}^2\,,
\end{align*}
which gives
\bgl{bernst6} \bignorm{f_0^{(s)}}_{L^2(\R)} \le \bignorm{f^{(s)}}_{L^2(\R)}\,.
\egl
Now, combining \eqref{bernst4}--\eqref{bernst6}, we conclude that
\[
\|f^{(s)}\|_{L^2(\R)}\le \|f_0^{(s)}\|_{L^2(\R)} + \|f_1^{(s)}\|_{L^2(\R)} \le \sigma^s \|f\|_{L^2(\R)} + \dist_2(f^{(s)}, B_\sigma^2)
\]
as was to be shown.
\end{proof}

\begin{corollary}\label{bernst_Sob}
Let $m,s\in\N$ with $m\ge s$ and $f\in W^{m,2}(\R)\cap F^{s,2}$.\smallskip

\noindent a) If $f^{(m)}\in \Mn$, then
\[
     \|f^{(s)}\|_{L^p(\R)} \le \sigma^s \|f\|_{L^p(\R)}+ \Oh(\sigma^{-1/2-m+s})\qquad(\sigma\to\infty).
\]

\noindent
b) If $f^{(m)}\in \Rieszalpha$, then for $\alpha >0$,
\[
     \|f^{(s)}\|_{L^p(\R)} \le \sigma^s \|f\|_{L^p(\R)}+ \oh(\sigma^{-\alpha-m+s})\qquad(\sigma\to\infty).
\]
\end{corollary}

When replacing the space $\Mn$ by $\Rieszhal$, then the order $\Oh(\sigma^{-1/2-m+s})$ is improved to $\oh(\sigma^{-1/2-m+s})$. More generally, for $\alpha < \frac12$, the order in a) is better than the order in b), whereas far $\alpha\ge 1/2$ the order in b) is better.

\begin{proof} We only need an estimate for $\dist_2(f^{(s)}, B_\sigma^2)$.
Part~a) follows from Proposition~\ref{dist_M_star}\,b) for $q=2$ and $r=m$. For b) we rewrite Proposition~\ref{prop_Riesz_and_oh} in terms of the distance functional, namely,
\begin{equation}\label{eq_H_implies_dist}
  f^{(m)}\in \Rieszalpha \implies \dist_2(f,B_\sigma^2)=\oh\pbig{\sigma^{-\alpha-m}}\qquad (\sigma\to\infty).
\end{equation}
Applying this  implication to $g=f^{(s)}$ and replacing $m$ by $m-s$ yields
\begin{align*}
  f^{(m)}=g^{(m-s)}\in \Rieszalpha%\\[2ex]
  {}\implies{} \dist_2(f^{(s)},B_\sigma^2)=\dist_2(g,B_\sigma^2)=\oh\pbig{\sigma^{-\alpha-m+s}}\qquad(\sigma\to\infty).%\qedhere
\end{align*}
\end{proof}

% -------------------------------------------------------------------------------------------------------------------------

\subsection{Nikol'ski{\u{\i}}'s inequality}
Nikol'ski{\u{\i}} \cite[pp.~123--124]{Nikolskii-1975} proved:
\begin{trivlist}\item[]
\it Let $f\in B_\sigma^2$. Then for any $h>0$,
\begin{equation}\label{nikol0}
  \biggbr{h\sum_{k\in\Z} \abs{f(hk)}^2}^{1/2}\le(1+ h\sigma) \|f\|_{\Ltwo}.
\end{equation}
\end{trivlist}
This result follows easily from Proposition~\ref{propw2} with the help of Bernstein's inequality.

The extension to non-bandlimited functions is just the assertion of Proposition~\ref{propw2}. Combining now this proposition with Theorem~\ref{bernst_L2}, we obtain the following statements.

\begin{theorem}\label{Nikol}
Let $f\in W^{1,2}(\R)\cap F^{1,2}$. Then, for any $h>0$ and $\sigma>0$, we have
\begin{equation*}
 \bigg\{h\sum_{k\in\Z} \abs{f(hk)}^2\bigg\}^{1/2}\le (1+ h\sigma) \|f\|_{L^2(\R)} + h \dist_2(f', B_\sigma^2)\,.
\end{equation*}
\end{theorem}

With the help of Proposition~\ref{dist_M_star}, we obtain the following corollary.

\begin{corollary}\label{Nikol_Sob}
Let $f\in W^{m,2}(\R)\cap F^{1,2}$ for some $m\in\N$.\smallskip

\noindent a) If $f^{(m)}\in \Mn$, then
\[
     \bigg\{h\sum_{k\in\Z} \abs{f(hk)}^2\bigg\}^{1/2} \le (1+ h\sigma)\|f\|_{L^2(\R)}+ h\Oh(\sigma^{-m+1/2})\qquad(\sigma\to\infty).
\]

\noindent
b) If $f^{(m)}\in \Rieszalpha$, then for $\alpha >0$,
\[
     \bigg\{h\sum_{k\in\Z} \abs{f(hk)}^2\bigg\}^{1/2} \le (1+ h\sigma)\|f\|_{L^2(\R)}+ h\oh(\sigma^{-\alpha-m+1})\qquad(\sigma\to\infty).
\]
\end{corollary}

Comparing the orders in a) and b), one sees right off that for $\alpha\ge \frac 12$ the order in b) is better than in a) but for $0<\alpha <\frac12$ it is weaker.

\subsection{The reproducing kernel formula}
The well known reproducing kernel formula in Bernstein spaces (see \eg \cite{Butzer-Ferreira_et_al_2011b}) reads:
\begin{trivlist}\item[]
{\it If $f\in B_\sigma^2$, then, for $0<h\le \pi/\sigma,$ \bgl{RKF1} f(z) = \rez{h}\int_\R f(u)\,\sinc(h^{-1}(z-u))\du \qquad (z\in\C). \egl }
\end{trivlist}
Its extension to non-bandlimited functions was first studied in \cite{Butzer_et_al_2014}, thus:

\begin{theorem}
Let $f\in F^2\cap S^2_h$ for some $h>0$. Then \bgl{RKFtwo1} f(t)\,=\, \frac1h\int_\R f(u)\sinc(h^{-1}(t-u))du + (\Rhf{RKF})(t) \qquad (t\in\R),
\egl
where the remainder $\Rhf{RKF}$ is expressed in terms of the remainder $\Rhf{WKS}$ for the approximate
sampling formula \eqref{ASF1} by
\bgl{RKFtwo2}
  (\Rhf{RKF})(t) = (\Rhf{WKS})(t) -\frac1h \int_\R (\Rhf{WKS})(u)\sinc(h^{-1}(t-u))\du\,.
\egl
Furthermore,
\bgl{RKFtwo3}
 \bigabs{(\Rhf{RKF})(t)} \le \frac{1}{\sqrt{2\pi}} \dist_1(f,B^2_{\pi/h}). %+\rez{\sqrt{h}}\big\|\Rhf{WKS}\|_{L^2(\R)}\,.
\egl
\end{theorem}

For the quite simple proof see \cite{Butzer-Schmeisser-Stens_2013}. Estimate \eqref{RKFtwo3} is sharp in the same sense as explicated for \eqref{ASF3}. An extremal function is $\sinc(2h^{-1}t)$.

\begin{corollary}\label{RKF_Sob}
If $f\in W^{m,2}(\R)\cap C(\R)$ and $f^{(m)}\in \Mn$ for some $m\in\N$, then
\[
     |(\Rhf{RKF})(t)|=\Oh(h^{m})\qquad (h\to 0+).
\]
\end{corollary}

%------------------------------------------------------------------------------------------------------------------------------

\subsection{The general Parseval formula}
The well-known Parseval formula for bandlimited functions $f\in B_\sigma^2$ and $h\in (0,\pi/\sigma]$, stating that
\bgl{PDF1}
   \int_\R \abs{f(u)}^2\,du = h \sum_{k\in\Z} \abs{f(hk)}^2,
\egl
is known to have the following generalization:
\begin{trivlist}\item[]
\it Let $f, g \in B_\sigma^2.$ Then, for $h\in (0, \pi/\sigma)$, we have
\[
  \int_\R f(u) \ol{g(u)}\du = h \sum_{k\in\Z} f(hk) \ol{g(hk)}\,.
\]
\end{trivlist}
An extension beyond bandlimited functions was established in \cite[Theorem~1.1\,(b)]{Butzer-Gessinger_1997} and was there called \emph{general Parseval
decomposition formula}. It may be stated and supplemented as follows.

\begin{theorem}\label{GPDF}
a) Let $f\in F^2\cap S_h^1$ and $g\in F^2$. Then
\begin{equation}\label{GPDF0}
   \int_\R f(u) \ol{g(u)}\,du =  h \sum_{k\in\Z} f(hk) \ol{g(hk)} + \Rh(f,g)\,,
\end{equation}
where
\bgl{GPDF1}
    \Rh(f,g) :=  \int_\R (\Rhf{WKS})(u) \ol{g(u)}\du - h \sum_{k\in\Z} f(hk)\, \rez{\sqrt{2\pi}} \int_{\abs{v}\ge \pi/h} \wh{\ol{g}}(v)\,e^{ihkv}\dv\,.
\egl

\noindent
b) Let $f, g \in W^{1,2}(\R)\cap F^{1,2}$. Then for any $h>0$ and a constant $C$ independent of $h$,
\begin{align}
\bigabs{\Rh(f,g)} {}\le{}& C\,h \Big\{\dist_2(f', B_{\pi/h}^2)+\dist_2(g', B_{\pi/h}^2) \notag \\[1ex]
&\qquad\qquad\qquad  {} + h\dist_2(f', B_{\pi/h}^2)\dist_2(g', B_{\pi/h}^2)\Big\}.
\label{PDF_Sob0}
\end{align}
\end{theorem}
For the proof, which depends on several preliminary estimates and is quite extensive, the reader is referred to \cite{Butzer-Schmeisser-Stens_2013}.

Concerning the orders for $h\to 0+$ in \eqref{PDF_Sob0}, one has as a result of the theory developed in this paper:

\begin{corollary}\label{cor_PDF}
Let $f\in W^{m,2}(\R)\cap F^{1,2}$ for some $m\in\N$.\smallskip

\noindent a) If $f^{(m)}, g^{(m)}\in \Mn$, then
\[
     \bigabs{(\Rhf{PDF})(t)}= \Oh(h^{m+1/2})\qquad(h \to 0+).
\]

\noindent
b) If $f^{(m)}, g^{(m)}\in \Rieszalpha$, then for $\alpha >0$,
\[
     \bigabs{(\Rhf{PDF})(t)}= \oh(h^{m+\alpha})\qquad(h\to 0+).
\]
\end{corollary}

The proof is again based on Proposition~\ref{dist_M_star}\,b) with $q=2$ and $s=1$.

Although we know that the estimates of Proposition~\ref{dist_M_star}\,b) for the distance functionals occurring on the left-hand side of \eqref{PDF_Sob0} are best possible, it is still an open question whether the order for $\Rhf{PDF}$ itself in Corollary~\ref{cor_PDF}\,a) is best possible or if it can be improved to $\oh(h^{m+1/2})$. This is due to the inequality in \eqref{PDF_Sob0}. A counterexample with functions $f^{(m)}, g^{(m)} \in \Mn$ with equality in \eqref{PDF_Sob0} would answer this question.

The reader may ask why we do not have estimates for functions belonging to fractional Sobolev spaces in Corollaries~\ref{ASF_Sob} and \ref{RKF_Sob}. This is due to the fact that the remainder terms for the Whittaker-Kotel’nikov-Shannon sampling theorem in \eqref{ASF3} and for the reproducing kernel theorem in \eqref{RKFtwo3} are estimated by the distance functional $\dist_1$ with respect to $\Lone$-norm, whereas Corollaries~\ref{bernst_Sob}, \ref{Nikol_Sob} and \ref{cor_PDF} are based on the estimate of the distance functional $\dist_2$ with respect to $\Ltwo$-norm in Proposition~\ref{prop_Riesz_and_oh}, rewritten in the form \eqref{eq_H_implies_dist}. Such an estimate is not available for $\dist_1$, since we have restricted the theory of Riesz derivatives to $\Ltwo$. For Riesz derivatives in $L^p(\R)$, $1\le p\le 2$, see \eg \cite[Chapter~11]{Butzer-Nessel_1971}.

The applications considered in this section, such as the sampling theorem, Bernstein's inequality, the reproducing kernel formula or Parseval's formula, all in the case of non-bandlimited functions, cannot be treated in the frame of the classical  $M^{2,1}$  space. However, as seen above, they can be carried out in $\Mn$.

The treatment of these applications in the more abstract setting of reproducing kernel Hilbert spaces does not seem to be possible, but reproducing kernel theory in the setting of Banach spaces, developed only in the last 20  years, could surly  be a suitable tool (see the overview in \cite[Section~6]{Butzer_et_al_2014} and the literature cited there).

We are of the firm opinion that in a generalization of the present $L^2(\R)$ theory to a Banach reproducing kernel setting, the arising adapted modulation space will render possibly  not only more theoretical but especially more practical applications, such as those of the present paper, but in a general Banach space setting. The various inclusion relations between Lipschitz and (fractional) Sobolev spaces should be of general interest in functional analysis.

%---------------------------------------------------------------------------------------------------------------------------------------------

\section{Counterexamples concerning the inclusions between various spaces}\label{sec_counterexamples}

In order to show that certain inclusions of spaces are strict, we shall employ counter-examples of the form
\[
 f(t) = \sum_{n=1}^\infty a_n\,\sinc^2(b_nt)\,e^{ic_nt}.
\]
By choosing the sequences $(a_n)_{n\in\N}$, $(b_n)_{n\in\N}$ and $(c_n)_{n\in\N}$ appropriately, we can design functions with very particular properties.

For calculating the Fourier transform $\widehat{f}$ and its norms, we shall use the following  lemma which is easily verified;
cf.~\cite[pp.~515--516, Table~1]{Butzer-Nessel_1971}.

\begin{lemma}\label{lemma-example}
For real numbers $b>0$ and $c$, let $\varphi(t):=\sinc^2(bt) e^{ict}$. Then, defining $I:=[c-2\pi b, c+2\pi b]$, we have:
\begin{align*}
\widehat{\varphi}(v)= &
\begin{cases}
\ds\rez{\sqrt{2\pi}\, b} \pbigg{1- \frac{\abs{v-c}}{2\pi b}}, & v\in I,\\[1.5ex]
\ds 0, & v\in\R\setminus I;
\end{cases}
\end{align*}
\begin{equation}\label{eq_phi_norms}
  \bignorm{\mspace{1mu}\widehat{\varphi}\mspace{1mu}}_{L^1(\R)} = \bignorm{\mspace{1mu}\widehat{\varphi}\mspace{1mu}}_{L^1(I)} = \sqrt{2\pi};\qquad
   \bignorm{\mspace{1mu}\widehat{\varphi}\mspace{1mu}}_{L^2(\R)} = \bignorm{\mspace{1mu}\widehat{\varphi}\mspace{1mu}}_{L^2(I)} = \sqrt{\frac{2}{3b}}\,.
\end{equation}
\end{lemma}

\begin{proposition}\label{prop_counter_allgemein}
For $\gamma>1$ and $\delta\ge 0$ let the function $f_{\gamma,\delta}$ be given by
\begin{equation}\label{eq_f_gamma_delta}
   f_{\gamma,\delta}(t) := \sum_{n=2}^\infty \rez{n^{\gamma}\,\ln^\delta n} \sinc^2\left(\frac{t}{2\pi n}\right) e^{i(n+1/2)t}.
\end{equation}
There holds
\begin{aufzaehlung}{$(iii)$}
\item $\ds f_{\alpha+1,1/2}\in \pbig{\lip_r(\alpha)\setminus \Rieszalpha}\cap C(\R)
                                                   %\subset \lip_r(\alpha)\setminus \Riesz(\alpha)
                                                   \quad (\alpha>0)$;

\item $\ds f_{\alpha+1,0}\in \pbig{\Lip_r(\alpha)\setminus \lip_r(\alpha)}\cap C(\R)
                                                  %\subset \Lip_r(\alpha)\setminus \lip_r(\alpha)
                                                  \quad (\alpha>0)$;

\item $\ds f_{\beta+1,0}\in \pbig{\Lip_r(\beta)\setminus \Lip_r(\alpha)}\cap C(\R)
                                         %\subset \Lip_r(\beta)\setminus \Lip_r(\alpha)
                                         \quad (\alpha>\beta>0)$;

\item $\ds f_{3/2,1} \in \pbig{F^2\cap S_h^2}\setminus M^{2,1}\quad (h>0)$.
\end{aufzaehlung}
\end{proposition}

\begin{proof}
First we note that $f_{\gamma,\delta}\in C(\R)$, in view of the uniform convergence of the series \eqref{eq_f_gamma_delta}. Next we show that the series converges also in $\Ltwo$-norm. Defining
\[
      \varphi_n(t) :=\sinc^2\left(\frac{t}{2\pi n}\right) e^{i(n+1/2)t},
\]
it follows by Lemma~\ref{lemma-example} that the support of $\widehat{\varphi}_n$ is the interval $I_n:= [n+\rez{2}-\rez{n}, n+\rez{2}+\rez{n}]$, and $I_{n_1}\cap I_{n_2} = \emptyset$ when $n_1$ and $n_2$ are any two different integers greater than $1$. This yields, noting the right-hand equation in \eqref{eq_phi_norms},
\[
  \int_\R \varphi_m(t) \overline{\varphi_n(t)} \dt = \int_\R \wh\varphi_m(t) \overline{\wh\varphi_n(t)} \dt=
  \begin{cases}
    \dfrac{4\pi n}{3}, & m=n\\[1.5ex]
    0, & m\neq n
  \end{cases}\qquad (m,n\ge 2).
\]
Hence the sequence $\pbig{\sqrt{3/4\pi n}\,\varphi_n}_{n\ge 2}$ is an othonormal sequence in the Hilbert space $\Ltwo$, and by the Riesz-Fischer theorem (cf.~\cite[p.~86]{Taylor-Lay_1980}), a series $\sum_{n=2}^\infty a_n \sqrt{3/4\pi n}\,\varphi_n$ converges in $\Ltwo$ if and only if $\sum_{n=2}^\infty\abs{a_n}^2<\infty$. This implies that the series \eqref{eq_f_gamma_delta} converges in $\Ltwo$, and, in particular, $f_{\gamma,\delta}\in \Ltwo$.

In view of the continuity of the Fourier transform we now have
\[
    \widehat{f_{\gamma,\delta}}(v) = \sum_{n=2}^\infty \rez{n^{\gamma}\,\ln^\delta n}\, \widehat{\varphi}_n(v),
\]
and Lemma~\ref{lemma-example} easily yields
\begin{align}
\bignorm{\mspace{1mu}\widehat{f_{\gamma,\delta}}\mspace{1mu}}_{L^1(\R)} \le {} & \sqrt{2\pi} \,\sum_{n=2}^\infty \rez{n^{\gamma}\,\ln^\delta n} < \infty
\label{eq_f_gamma_1_norm}
%\\[2ex]
%%
%%\bignorm{\mspace{1mu}\widehat{f_{\gamma,\delta}}\mspace{1mu}}_{L^2(\R)} = {} &
%%                       2\, \sqrt{\frac{\pi}{3}}\brbigg{\sum_{n=2}^\infty \rez{n^{2\gamma-1} \,\ln^{2\delta} n}}^{1/2} < \infty\label{eq_f_gamma_2_norm}
%
\end{align}
and
\begin{equation}\label{eq_f_gamma_modulation_norm_part}
  \bignorm{\mspace{1mu}\widehat{f_{\gamma,\delta}}\mspace{1mu}}_{L^2[n, n+1]}^2 =
  \begin{cases}\ds \frac{4\pi}{3}  \rez{n^{2\gamma-1}\,\ln^{2\delta} n}, & n\ge 2\\[2ex]
                0, & n\le 1
  \end{cases} \qquad (n\in\Z).
\end{equation}

For $\sigma \ge 3$ and $m\in\N$ with $m\le \sigma < m+1$ we have
\begin{align*}
  & \int_{\abs{v}\ge \sigma}^\infty \bigabs{\widehat{f_{\gamma,\delta}}(v)}^2dv
      \le \sum_{n=m}^\infty  \bignorm{\mspace{1mu}\widehat{f_{\gamma,\delta}}\mspace{1mu}}_{L^2[n, n+1]}^2
       = \frac{4\pi}{3} \sum_{n=m}^\infty \rez{n^{2\gamma-1}\,\ln^{2\delta} n} \\[2ex]
  {}\le{} & \frac{4\pi}{3\ln^{2\delta}m}\int_{m-1}^\infty\frac{du}{u^{2\gamma-1}}=\frac{4\pi(m-1)^{-2\gamma+2}}{3(2\gamma-2)\ln^{2\delta}m}%\\[2ex]
  {}\le{} \frac{4\pi(\sigma-2)^{-2\gamma+2}}{3(2\gamma-2)\ln^{2\delta}(\sigma-1)}.
\end{align*}
This shows that
\begin{equation}\label{eq_counterex_Lipschitz}
   \int_{\abs{v}\ge \sigma}^\infty \bigabs{\widehat{f_{\gamma,\delta}}(v)}^2dv
   =\Oh\pbig{{\sigma^{-2\gamma+2}\ln^{-2\delta}\sigma}}\qquad (\sigma\to\infty).
\end{equation}

An estimate of this integral from below will be needed for $\delta = 0$ only. If $\sigma$ and $m$ are as above, then
\begin{align}
   \int_{\abs{v}\ge \sigma}^\infty \bigabs{\widehat{f_{\gamma,0}}(v)}^2dv
      \ge &\sum_{n=m+1}^\infty  \bignorm{\mspace{1mu}\widehat{f_{\gamma,0}}\mspace{1mu}}_{L^2[n, n+1]}^2
       = \frac{4\pi}{3} \sum_{n=m+1}^\infty \rez{n^{2\gamma-1}} \notag\\[-2ex]
       \phantom{f}\label{eq_counterex_Lipschitz_below}\\[0ex]
  {}\ge{} & \frac{4\pi}{3}\int_{m+2}^\infty\frac{du}{u^{2\gamma-1}}=\frac{4\pi(m+2)^{-2\gamma+2}}{3(2\gamma-2)}
  {}\ge{} \frac{4\pi(\sigma+2)^{-2\gamma+2}}{3(2\gamma-2)}.\notag
\end{align}

Now to the proof of assertions (i)--(iv).

\noindent(i): It follows from Theorem~\ref{thm_oh} and \eqref{eq_counterex_Lipschitz} that $f_{\alpha+1,1/2}\in \lip_r(\alpha)$. On the other hand, with the interval $I_n$ as defined above, we have
\begin{align*}
 \int_{-\infty}^\infty \abs{v}^{2\alpha}\bigabs{\widehat{f_{\alpha+1,1/2}}(v)}^2dv
  {}={} &\sum_{n=2}^\infty\int_{I_n} \abs{v}^{2\alpha}\bigabs{\widehat{f_{\alpha+1,1/2}}(v)}^2dv \\[2ex]
  {}\ge{}&  \sum_{n=2}^\infty n^{2\alpha} \bignorm{\widehat{f_{\alpha+1,1/2}}}_{L^2[n, n+1]}^2
  =  \frac{4\pi}{3}\sum_{n=2}^\infty \frac{1}{n \ln n} = \infty,
\end{align*}
which shows that $f_{\alpha+1,1/2}\notin \Rieszalpha$.

\smallskip\noindent
(ii), (iii): This follows immediately from the estimates \eqref{eq_counterex_Lipschitz} and \eqref{eq_counterex_Lipschitz_below}.

\smallskip\noindent
(iv): Equations \eqref{eq_f_gamma_1_norm} and \eqref{eq_f_gamma_modulation_norm_part} show that $f_{3/2,1} \in F^2\setminus M^{2,1}$.
Next, let $h>0$ and let $k$ be any integer such that $\abs{k}>80$. Then
\[
    \bigabs{f(hk)} \le \sum_{n=2}^\infty \rez{n^{3/2}\,\ln n} \sinc^2\pbigg{\frac{h\abs{k}}{2\pi n}},
\]
and with the estimate% \eqref{eq_sinc_estimate}
%Employing the estimate
%%
\[
   \sinc^2(t) \le \min\brBig{1, \rez{\pi^2 t^2} }\quad (t\in \R\setminus \{0\}),
\]
we find that
%
%\begin{align*}
%%
%\bigabs{f(hk)} \le{} & \frac{4}{h^2 k^2}\, \sum_{n=2}^{\abs{k}}\frac{n^{1/2}}{\ln n} +
%                                                        \sum_{n= \abs{k}+1}^\infty \rez{n^{3/2}\,\ln n}\\[2ex]
%%
%{}\le {} & \frac{4}{h^2 \abs{k}}\, \max_{2\le n\le\abs{k}} \frac{n^{1/2}}{\ln n} + \rez{\ln \abs{k}} \int_{\abs{k}}^\infty \frac{dx}{x^{3/2}}\\[2ex]
%%
%{} = {}  & \frac{4}{h^2 \abs{k}^{1/2}\,\ln \abs{k}}\,+\, \frac{2}{\abs{k}^{1/2}\,\ln \abs{k}}\\[2ex]
%%
%{} = {} & 2\pBig{1 +\frac{2}{h^2}} \rez{\abs{k}^{1/2}\,\ln\abs{k}},
%%
%\end{align*}
%
\begin{align*}
 \bigabs{f(hk)} \le{} & \frac{4}{h^2 k^2}\, \sum_{n=2}^{\abs{k}}\frac{n^{1/2}}{\ln n} +
                                                        \sum_{n= \abs{k}+1}^\infty \rez{n^{3/2}\,\ln n}\\[2ex]
{}\le {} & \frac{4}{h^2 \abs{k}}\, \max_{2\le n\le\abs{k}} \frac{n^{1/2}}{\ln n} + \rez{\ln \abs{k}} \int_{\abs{k}}^\infty \frac{dx}{x^{3/2}}\\[2ex]
{} = {} &  \frac{4}{h^2 \abs{k}^{1/2}\,\ln \abs{k}}\,+\, \frac{2}{\abs{k}^{1/2}\,\ln \abs{k}}%\\[2ex]
{} = {}  2\pBig{1 +\frac{2}{h^2}} \rez{\abs{k}^{1/2}\,\ln\abs{k}},
\end{align*}
where we have used that the maximum in the second inequality is attained at $n=\abs{k}$ when $\abs{k}>80$, as can be verified by elementary calculus. Hence $f\in S_h^2$.
\end{proof}

\begin{proposition}\label{prop_counterex2}
Let
\[
   f(t) := \sinc^2\left(\frac{t}{4\pi}\right) \sum_{n=1}^\infty \frac{e^{i(2^n-1/2)t}}{n^2}.
\]
Then the following holds:
\begin{aufzaehlung}{$(iii)$\hfil}
\item  $\ds f\in M^{2,1}\setminus \Mn$;
\item $\ds \dist_2(f, B_\sigma^2)\,\ge\, \frac{2\sqrt{2\pi}}{3} \pBig{\frac{\ln2}{3 \ln \sigma}}^{3/2} \quad (\sigma \ge 2)$;
\item for $r\in\N$ and any $\alpha\in (0, r)$, we have $f\not\in \Lip_r(\alpha)$.
\end{aufzaehlung}
\end{proposition}

\begin{proof} Obviously $f\in\Ltwo\cap C(\R)$, and by Lemma~\ref{lemma-example} the Fourier transform of
\begin{equation}\label{def_psi}
\psi_n(t) :=  \sinc^2\left(\frac{t}{4\pi}\right) e^{i(2^n-1/2)t}
\end{equation}
has support
\begin{equation}\label{interval_J}
     J_n:=[2^n -1, 2^n]
\end{equation}
and
\[
   \bignorm{\widehat{\psi}_n}_{L^2(J_n)} = 2\,\sqrt{\frac{2\pi}{3}}\,.
\]
Therefore
\[\sum_{n\in\Z} \bignorm{\widehat{f}\,}_{L^2[n, n+1]}
 % \sum_{n\in\Z}\brbigg{\int_n^{n+1} \bigabs{\wh{f}(v)}^2\dv}^{1/2}
  = \sum_{n=1}^\infty \frac1{n^2}\, \bignorm{\widehat{\psi}_n}_{L^2(J_n)}
   = \frac{\pi^2}{3}\, \sqrt{\frac{2\pi}{3}}\,,
\]
and so $f\in M^{2,1}$.

In order to show that $f\not\in \Mn$, it suffices to prove that the series $\mathcal N_h(f)$ of \eqref{Nhf} is not uniformly bounded for $h\in (0, 1]$. Indeed, replace $h$ by $h_k:= 2^{-k}$ which lies in $(0, 1]$ for each $k\in\N_0$. Then
\begin{align*}
 \mathcal{N}_{h_k}(f) ={} &\sum_{n=1}^\infty \brbigg{2^k \int_{n2^k}^{(n+1)2^k} \bigabs{\wh{f}(v)}^2dv}^{1/2}
 \ge 2^{k/2} \sum_{n=k+1}^\infty\brbigg{\int_{J_n} \frac1{n^4}\,{\bigabs{\widehat{\psi}_n(v)}^2}dv}^{1/2}\\[2ex]
 {}={}& \sqrt{\frac{2\pi}{3}}\, 2^{k/2 +1} \sum_{n=k+1}^\infty \rez{n^2} \ge \sqrt{\frac{2\pi}{3}}\, \frac{2^{k/2 +1}}{k+1}\,.
\end{align*}
Clearly, the right-hand side does not remain bounded for $k\in \N$ and so $f\not\in \Mn.$ This completes the proof of statement (i).

Next, given $\sigma\ge 2$, we denote by $k_\sigma$ the uniquely determined integer such that
\begin{equation}\label{definition_k_sigma}
2^{k_\sigma -1} -1\, <\, \sigma \,\le\, 2^{k_\sigma} - 1.
\end{equation}
Then
\begin{align}
\pbig{\dist_2(f, B_\sigma^2)}^2 = {} & \int_{\abs{v}\ge\sigma} \bigabs{\widehat{f}(v)}^2\,dv   \ge  \sum_{n=k_\sigma}^\infty
         \frac1{n^4}\,{\bignorm{\widehat{\psi}_n}_{L^2(J_n)}^2}\notag\\[1.5ex]
{} = {} & \frac{8\pi}{3}\, \sum_{n=k_\sigma}^\infty \rez{n^4} \ge \frac{8\pi}{3} \int_{k_\sigma}^\infty \frac{dx}{x^4}
   = \frac{8\pi}{9 k_\sigma^3}\,\label{lower_bound}.
\end{align}
From the first inequality in \eqref{definition_k_sigma} we deduce that
$$ k_\sigma\,< \, \frac{\ln(\sigma+1) + \ln 2}{\ln 2}\, \le\,
\frac{ 3 \ln \sigma}{\ln 2}$$
for $\sigma\ge 2.$
Substituting this bound in (\ref{lower_bound}), we obtain statement (ii).

Finally, statement (iii) follows from (ii) by means of Theorem~\ref{thm_equiv_new}.
\end{proof}

\begin{remark}
In view of statement (ii) of Proposition~\ref{prop_counterex2} it should be observed that for a function $g\in M^{2,1}$ the convergence to zero of $\dist_2(g, B_\sigma^2)$ as $\sigma\to\infty$ can be arbitrarily slow. For the construction of a function with a prescribed (slow) convergence rate, we may simply modify the function $f$ in Proposition~\ref{prop_counterex2} by replacing $2^n$ in its definition by a sufficiently fast increasing sequence of natural numbers $m_n$.
\end{remark}

The following two propositions will serve us for comparing $\Rieszhal$ with $\Mn$.

\begin{proposition}\label{prop_counterex3}
The function $f$, given by
\begin{equation}\label{eq_counter_oh}
    f(t):= \sinc^2\pBig{\frac{t}{4\pi}} \sum_{n=1}^\infty \frac{e^{i(2^n-1/2)t}}{2^{n/2}},
\end{equation}
belongs to $\Mn\setminus \Rieszhal$.
\end{proposition}

\begin{proof}
Let $N\in\N$. Using the notation (\ref{def_psi}) and (\ref{interval_J}), we conclude that
\begin{align*}
& \int_0^{2^N} \abs{v}\cdot \bigabs{\widehat{f}(v)}^2\,dv =\sum_{n=1}^N \rez{2^n} \int_{J_n} \abs{v}\cdot\bigabs{\widehat{\psi}_n(v)}^2\,dv\\[2ex]
{}\ge{} & \frac{8\pi}{3} \sum_{n=1}^N \frac{2^n-1}{2^n} \ge \frac{8\pi}{3}(N-1) \, \longrightarrow \infty \quad (N\to\infty),
\end{align*}
and so $f\notin \Rieszhal$.

In order to prove that $f\in \Mn$, it suffices to show that $\mathcal N(f) <\infty$.
%
%$$\sum_{n=1}^\infty \left(\rez{h} \int_{n/h}^{(n+1)/h} \abs{\widehat{f}(v)}^2\,dv\right)^{1/2}$$
%
%is bounded for $h\in(0,1]$.
%
First we note that for $h\in(0,1]$ and $n\ge 1$ the interval $H_n:=(n/h, (n+1)/h]$ can contain at most one power of $2$. Let $n_1\ge 1$ be the smallest integer such that $H_{n_1}$ contains a power of $2$, say $2^{j_1}\in H_{n_1}$. Then $1/h< 2^{j_1}/n_1 \le 2^{j_1}$ and therefore
\begin{align*}
& \sum_{n=1}^\infty \brbigg{\rez{h} \int_{n/h}^{(n+1)/h} \bigabs{\widehat{f}(v)}^2\,dv}^{1/2}
  \le \sum_{j=j_1}^\infty\brbigg{\rez{h2^j} \int_{J_j}\bigabs{\widehat{\psi}_j(v)}^2\,dv}^{1/2}\\[2ex]
{}\le{}& \sum_{j=j_1}^\infty\left(\frac{2^{j_1}}{2^j}\cdot\frac{8\pi}{3}\right)^{1/2} = \frac{4}{\sqrt{2}-1}\, \sqrt{\frac{\pi}{3}}
\end{align*}
independently of $h$. This completes the proof.
\end{proof}

The function $f$ defined in \eqref{eq_counter_oh} above also enables us to prove Proposition~\ref{prop_oh}.

{\renewcommand{\proofname}{Proof of Proposition~\ref{prop_oh}}
\begin{proof}
Since $\omega_r(f;\delta; L^2(\R)) = \oh\pbig{\delta^{1/2}}$ is equivalent to $\dist_2(f,B^2_\sigma)= \oh\pbig{\sigma^{-1/2}}$ by Theorem~\ref{thm_oh}, it is enough to prove the assertion concerning the modulus of continuity.

Consider the function \eqref{eq_counter_oh}, which was shown to belong to $\Mn$ in Proposition~\ref{prop_counterex3}. Replacing $h$ by $h_k:= \pi 2^{-k}$ with $k\in\N$, we want to establish a lower bound for the left-hand side of \eqref{Delta_bounds}. With $\psi_n$ and $J_n$ as in the proof of Proposition~\ref{prop_counterex3}, we find that

\begin{align*}
& {}\li(\frac{2}{\pi}\ri)^{2r} \int_{\bigabs{v}\le \pi/h_k} \bigabs{\wh{f}(v)}^2 (h_kv)^{2r} dv
=  \li(\frac{2}{\pi}\ri)^{2r} \int_{\abs{v}\le 2^k} \bigabs{\wh{f}(v)}^2 \pbig{\pi 2^{-k}v}^{2r} dv \\[2ex]
{}={}& \li(\frac{2}{\pi}\ri)^{2r} \sum_{n=1}^k\frac{\pi^{2r}}{2^{2kr+n}} \int_{J_n} \bigabs{\wh{\psi}_n(v)}^2 v^{2r} dv
\ge \frac{8\pi}{3}\cdot \rez{2^{2r(k-1)}} \sum_{n=1}^k \rez{2^n} \li(2^n-1\ri)^{2r}\\[2ex]
{}={}& \: \frac{8\pi}{3}\cdot \rez{2^{2r(k-1)}} \sum_{n=1}^k 2^{(2r-1)n} \li(1-\rez{2^n}\ri)^{2r}
    \ge \frac{8\pi}{3}\cdot \rez{2^{2rk}} \sum_{n=1}^k 2^{(2r-1)n}\\[2ex]
{}={}& \frac{8\pi}{3}\cdot \frac{2^{2r-1}}{2^{2rk}}\cdot \frac{2^{(2r-1)k}-1}{2^{2r-1}-1}
    \ge \frac{8\pi}{3}\cdot \rez{2^k}\,=\, \frac{8}{3}\,h_k\,.
\end{align*}
Now the left-hand side of \eqref{Delta_bounds} gives
\[
  \li\|\Delta^r_{h_k} f\ri\|_{L^2(\R)}^2 \ge \frac{8}{3} h_k \qquad (k\in\N).
\]
This implies $\omega_r(f;\delta; L^2(\R)) \neq \oh\pbig{\delta^{1/2}}$.
\end{proof}}

In the next proposition, which concerns functions belonging to $\Rieszhal$ but not to $\Mn$, it is just the exponent $3/4$ of the logarithm in \eqref{eq_counter_ln} that does the job.

\begin{proposition}\label{prop_counterex4}
The function $f$, given by
\begin{equation}\label{eq_counter_ln}
     f(t) := \sinc^2\left(\frac{t}{4\pi}\right) \sum_{n=2}^\infty \frac{e^{i(n^2-1/2)t}}{n^{3/2}\,(\ln n)^{3/4}},
\end{equation}
belongs to $\pbig{\Rieszhal \cap F^2 }\setminus \Mn$.
\end{proposition}

\begin{proof}
By Lemma~\ref{lemma-example} we have
\[
  \bignorm{\wh{f}\mspace{4mu}}_{\Lone}=\sqrt{2\pi}\,\sum_{n=2}^\infty \frac{1}{n^{3/2}\,(\ln n)^{3/4}}<\infty,
\]
implying $f\in F^2$.  Also by Lemma~\ref{lemma-example}, the Fourier transform of
\[
    \chi_n(t):= \sinc^2\left(\frac{t}{4\pi}\right) e^{i(n^2-1/2)t}
\]
has support $K_n:=[n^2 -1, n^2]$ and $\|\widehat{\chi}_n\|_{L^2(K_n)} =  2\, \sqrt{2\pi/3}$.
Hence
\begin{align*}
& \int_\R \abs{v}\cdot\bigabs{\widehat{f}(v)}^2\dv
            = \sum_{n=2}^\infty \rez{n^3 (\ln n)^{3/2}} \int_{K_n} v \bigabs{\widehat{\chi}_n(v)}^2\,dv \\[1ex]
{}\le {} & \sum_{n=2}^\infty \rez{n (\ln n)^{3/2}} \left\|\widehat{\chi}_n\right\|_{L^2(K_n)}^2
  = \frac{8\pi}{3} \sum_{n=2}^\infty \rez{n (\ln n)^{3/2}} < \infty,
\end{align*}
which shows that $f\in \Rieszhal$.

Now we show that $f\not\in \Mn.$ Considering \eqref{Nhf} for $h=h_k:=k^{-2}$ with $k\in\N$, we have
\begin{equation}\label{uniform_convy}
  \mathcal{N}_{h_k}(f) = \sum_{n=1}^\infty\brbigg{k^2 \int_{nk^2}^{(n+1)k^2} \bigabs{\widehat{f}(v)}^2\dv}^{1/2}.
\end{equation}
We shall show that the right-hand side approaches infinity as $k\to\infty$.

If $\mu\in\N$ and
\begin{equation}\label{mu_bounds}
  nk^2  <  \mu^2  \le (n+1)k^2\,,
\end{equation}
then $K_\mu \subset [nk^2, (n+1)k^2]$ and so
\begin{equation}\label{lower_bound1}
  k^2 \int_{K_\mu} \bigabs{\widehat{f}(v)}^2 \dv = \frac{8\pi k^2}{3\mu^3 (\ln \mu)^{3/2}}
  \ge \frac{8\pi}{3}\cdot \rez{(n+1)^{3/2} k \bkbig{\ln\big((n+1)^{1/2} k\big)}^{3/2}}\,.
\end{equation}
A short reflection shows that there exist at least $m_{n,k}:= (n+1)^{1/2} k - n^{1/2} k -2$ different integers $\mu$ satisfying (\ref{mu_bounds}). If $k > 12$ and
\begin{equation}\label{n_bounds}
 1 \le n \le \frac{k^2}{36} - 1,
\end{equation}
then
\[
   m_{n,k} \ge \frac{k}{6(n+1)^{1/2}} \ge 1\,.
\]
With this estimate we obtain in view of (\ref{lower_bound1}),
\[
   k^2 \int_{nk^2}^{(n+1)k^2} \bigabs{\widehat{f}(v)}^2\dv \ge \frac{4\pi}{9}\cdot \rez{(n+1)^2 \bkbig{\ln\pbig{(n+1)^{1/2} k}}^{3/2}}
\]
for $n$ satisfying \eqref{n_bounds}. Thus, setting $N_k := \left\lfloor k^2/36 -1 \right\rfloor$ and recalling \eqref{uniform_convy}, we have
\begin{align*}
  & \mathcal{N}_{h_k}(f) \ge \sum_{n=1}^{N_k}\brbigg{k^2 \int_{nk^2}^{(n+1)k^2} \bigabs{\widehat{f}(v)}^2\dv}^{1/2}\\[2ex]
  {}\ge{}& \frac{2\sqrt{\pi}}{3} \, \sum_{n=1}^{N_k} \rez{(n+1) \bkbig{\ln\pbig{(n+1)^{1/2}k}}^{3/4}}
                                 \ge \frac{2\sqrt{\pi}}{3} \, \sum_{n=2}^{N_k+1} \rez{n\bkbig{\ln(nk)}^{3/4}}\\[2ex]
  {} \ge {} & \frac{2\sqrt{\pi}}{3} \,k \int_2^{k^2/36} \frac{dx}{kx \bkbig{\ln(kx)}^{3/4}}
                                  = \frac{2\sqrt{\pi}}{3} \, \int_{2k}^{k^3/36} \frac{du}{u[\ln u]^{3/4}}\\[2ex]
  {} \ge {} & \frac{2\sqrt{\pi}}{3}\, \bkbig{\ln(2k)}^{1/4} \int_{2k}^{k^3/36} \frac{du}{u \ln u}
         =  \frac{2\sqrt{\pi}}{3} \,\bkbig{\ln(2k)}^{1/4} \bkbig{\ln(\ln u)}_{u=2k}^{u=k^3/36}.
\end{align*}
The right-hand side approaches infinity as $k\to\infty$.
\end{proof}

Next we compare $\Rieszhal\cup \Mn$ with $\Lip_r(\hal)\cap F^2$.

\begin{proposition}\label{prop_counterex5}
The function $f$, given by
\[
  f(t) := \sinc^2\left(\frac{t}{4\pi}\right) \sum_{n=2}^\infty \frac{e^{i(n^2-1/2)t}}{n^{3/2}},
\]
belongs to $\bigl(\Lip_r(\hal)\cap F^2\bigr)\setminus \bigl(\Rieszhal\cup \Mn\bigr)$.
\end{proposition}

\begin{proof}
The proof that $f\in \Lip_r(\hal)\cap F^2$ is quite straightforward and may be left to the reader. For showing that $f\not\in \Rieszhal\cup \Mn$ one only needs slight modifications of the previous proof.
\end{proof}

For the following comparisons of spaces we use functions $f_\gamma\in \Ltwo$, $\gamma>1/2$, defined via their Fourier transforms by
\begin{equation*}%\label{eq_f_gammax}
  \wh f_\gamma(v):=
  \begin{cases}
    0,&\D \abs{v}<1,\\[1ex]
    \D\abs{v}^{-\gamma},&\abs{v}\ge 1.
  \end{cases}
\end{equation*} .

\begin{proposition}\label{prop_counter_last}
a) Let $0 < \beta < \alpha$ and $1/2 + \beta <\gamma < 1/2+\alpha$. Then, $f_\gamma$ belongs to $\Rieszbeta\setminus \Lip_r(\alpha)$.
If, in addition, $\beta \ge 1/2$, then there even holds $f_\gamma \in \pbig{\Rieszbeta\setminus \Lip_r(\alpha)}\cap C(\R)$.

\medskip\noindent
b) Let $\alpha >1/2$ and $1<\gamma < 1/2+\alpha$. Then $f_\gamma\in \pbig{\Rieszhal \cap \Mn} \setminus \Lip_r(\alpha)$.
\end{proposition}

\begin{proof}
a) follows from the definition of $\Rieszbeta$ and Theorem~\ref{thm_equiv_new}\,(i)$\,\Leftrightarrow\,$(ii). For the second assertion in  a), one has additionally to note that the $f_\gamma$ can be chosen as a continuous function, since $\wh f_\gamma\in\Lone$ for $\gamma>1$.

Concerning b), one uses part a) for $\beta=1/2$ and the definition of $\Mn$.
\end{proof}

\section{Appendix}
In our joint paper \cite{Butzer-Schmeisser-Stens_2014a}, we had introduced  the modified modulation space $M^{2,1}_*$, also a subspace of the classical modulation space $M^{2,1}$;
%introduced by H.\,G.~Feichtinger \cite{Feichtinger_1981}
see \eqref{hans}, \eqref{eq_classical_modulation_norm}.
%
%To facilitate the readers' understanding, let us recall its definition. If
%%
%\begin{gather*}
%%
%   M^{2,1}=M^{2,1}(\R) := \li\{\strut f\;:\; f:=\wh{g},\, \, g\in W(L^2, \ell^1)\ri\}\\[2ex]
%%
%    \|f\|_{M^{2,1}} :=
%     \sum_{n\in\Z}\biggbr{\int_n^{n+1} \bigabs{\wh{f}(v)}^2\,dt}^{1/2}
%   = \Bignorm{\pBig{\bignorm{\wh f}_{L^2(n,n+1)}}_{n\in\Z}}_{\ell^1}.
%\end{gather*}
%
%is the classical modulation space, then
We defined $M^{2,1}_\ast$ as the set of all functions $f\in M^{2,1}$ such that the series

\begin{equation}\label{scaled_mod1}
  \sum_{n\in\Z} \rez{h} \biggbr{\int_n^{n+1} \abs{\wh{f}\li(\frac{v}{h}\ri)}^2\, dv}^{1/2}
  = \sum_{n\in\Z}  \biggbr{\rez{h}\int_{n/h}^{(n+1)/h} \bigabs{\wh{f}(v)}^2\,dv}^{1/2} < \infty.
\end{equation}
converges uniformly with respect to $h$ on bounded subintervals of $(0,\infty)$.

The connection between the three spaces $M^{2,1}$, $M^{2,1}_\ast$ and the readapted modulation space $\Mn$ is given in the following proposition.

\begin{proposition}\label{M_new_prop1}
We have $M^{2,1}_\ast \subsetneqq \Mn\subsetneqq M^{2,1}\subsetneqq F^2\cap S^2_h$ for all $h>0$.
\end{proposition}

\begin{proof}
Let $f\in M^{2,1}_\ast$. By the uniform convergence of the series \zit{scaled_mod1} on bounded subintervals of $(0, \infty)$, there exists an integer $n_0 > 1$ such that

\begin{equation}\label{M_new4}
\sum_{n\in\Z\setminus N_0} \sm{} \le \bignorm{\wh{f}\,}_{\Ltwo}
\end{equation}
for all $h\in(0,1]$, where $N_0:=\{-n_0, \dots , n_0-1\}.$

Now let $n\in\{1, \dots , n_0-1\}$ and define $\delta:= n_0 h/n$. If $\delta \ge 1$, then $1/h\le n_0$ and so
\begin{equation}\label{M_new5}
\sm{} \le n_0^{1/2} \bignorm{\wh{f}\,}_{\Ltwo}.
\end{equation}
If $\delta < 1$, then
\begin{align*}
& \sm{} = \brbigg{\frac{n_0}{n\delta} \int_{n_0/\delta}^{n_0/\delta + 1/h} \bigabs{\wh{f}(v)\,}^2 \dv}^{1/2}\\[2ex]
{}\le{} &  n_0^{1/2} \sum_{\ell=n_0}^\infty\brbigg{\rez{\delta}\int_{\ell/\delta}^{(\ell+1)/\delta} \abs{\wh{f}(v)}^2 \dv}^{1/2}
\le n_0^{1/2}\, \bignorm{\wh{f}\,}_{\Ltwo},
\end{align*}
where we have used \zit{M_new4} with $h$ replaced by $\delta$ in the last step. This is again the bound \zit{M_new5}. By proceeding analogously, we
obtain the bound \zit{M_new5} for $n\in\{-n_0, \dots , -2\}$ as well. Combining \zit{M_new4} with the bounds for $n\in N_0\setminus\{-1 , 0\}$,
we find that
\[
     \mathcal{N}(f) \le \li( 2n_0^{3/2} - 2n_0^{1/2}+1\ri) \bignorm{\wh{f}\,}_{\Ltwo}.
\]
This shows that $f\in \Mn$. Concerning the last two inclusion relations recall Proposition~\ref{propa1_star}. For $M^{2,1}_\ast \ne \Mn$ see the counterexample in the following proposition.
\end{proof}

\begin{proposition}\label{prop_M-M}
The function $f$, given by
\[
   f(t) := \sinc^2\left(\frac{t}{4\pi}\right) \sum_{n=2}^\infty\frac{e^{i(n^2-1/2)t}}{n^{3/2}\,\ln n},
\]
belongs to $\Mn\setminus M^{2,1}_\ast$.
\end{proposition}

\begin{proof}
By Lemma~\ref{lemma-example} the Fourier transform of
\[
   \chi_n(t) := \sinc^2\left(\frac{t}{4\pi}\right) e^{i(n^2-1/2)t}
\]
has support $K_n:=[n^2 -1, n^2]$ and $\|\widehat{\chi}_n\|_{L^2(K_n)} = 2 \sqrt{{2\pi}/3}$.
%
%\[
%     \left\|\widehat{\chi}_n\right\|_{L^2(K_n)}\,=\, 2\, \sqrt{\frac{2\pi}{3}}\,.
%\]
With this, we easily see that $f\in M^{2,1}$.

For proving $f\in \Mn$, we have to show that $\Nh(f)$, defined in \eqref{Nhf}, is bounded for $h\in(0,1]$. In the following considerations, the location of the intervals $K_\mu$ relative to the intervals $I_{n,h}:= [n/h, (n+1)/h]$ is crucial.

Let $\mu$ be an integer such that $\mu^2\in I_{n,h}$, that is,
\begin{equation}\label{proofM-M1}
   \frac{n}{h} \le \mu^2 \le \frac{n+1}{h}\,.
\end{equation}
Since
\begin{equation}\label{proofM-M2}
\sqrt{\frac{n+1}{h}} - \sqrt{\frac{n}{h}}\,<\, \rez{2\sqrt{nh}}\,,
\end{equation}
the number of intervals $K_\mu$ such that $K_\mu\cap I_{n,h}$ is of positive measure is less than $(2\sqrt{nh})^{-1}+1$; in particular, if the
right-hand side of \zit{proofM-M2} is less than $2$, or equivalently,
\bgl{proofM-M3}
nh  > \rez{16}\,,
\egl
then \zit{proofM-M1} can hold for at most one integer $\mu$.

Now, for given $h\in(0,1]$,  let $n_h$ be the smallest integer greater than $1$ that satisfies \zit{proofM-M3} and define
\[
   f_h(t) := \sum_{n=2}^{n_h} \frac{\chi_n(t)}{n^{3/2} \ln n}\,.
\]
Since $\Nh$ is a semi-norm on $M^{2,1}$, we may use the triangular inequality to conclude that
\[
   \Nh(f) \le \Nh(f_h) + \sum_{n=n_h+1}^\infty \rez{n^{3/2} \ln n} \Nh(\chi_n).
\]

Our next aim is to estimate the two terms on the right-hand side. If $K_\mu\cap I_{n,h}$ is of positive measure, then
\[
   \int_{K_\mu\cap I_{n,h}} \absbig{\wh{f}_h(v)}^2 dv\le \frac{\|\wh{\chi}_\mu\|_{L^2(K_\mu)}^2}{\mu^3 \ln^2 \mu} =  \frac{8\pi}{3 \mu^3 \ln^2 \mu}
   \le \frac{8\pi}{3} \pBig{\frac{h}{n}}^{3/2} \rez{\ln^2(\sqrt{n/h})}.
\]
Thus, for $n\in [2, n_h]$, we have
\[
  \int_{I_{n,h}} \bigabs{\wh{f}_h(v)}^2 dv \le \pBig{\rez{2\sqrt{nh}}+1}
                                  \frac{8\pi}{3}\pBig{\frac{h}{n}}^{3/2} \rez{\ln^2(\sqrt{n/h})}.
\]
Noting that the support of $\wh{f}_h$ is contained in $[3, (n_h+1)/h]$ and that $h\le 1/n_h$, we obtain
\[
  \Nh(f_h) \le \sum_{n=2}^{n_h} \brbigg{\rez{h} \int_{I_{n,h}} \absbig{\wh{f}_h(v)}^2 dv}^{1/2}
  \le 4\sqrt{\pi} \,\sum_{n=2}^{n_h} \rez{n \ln(n/h)} \le 4\sqrt{\pi} \,\frac{\ln n_h}{\ln(2/h)}\,.
\]
Since $n_h \le 2+ 1/(16h)$, we find that the right-hand side is bounded for $h\in(0,1]$.

Now we estimate $\Nh(\chi_n)$ for $n>n_h$. Note that the interval $K_n$ may either be a subset of one of the intervals $I_{\ell,h}$ or it may have a non-empty intersection with exactly two consecutive intervals $I_{\ell,h}$ and $I_{\ell+1,h}$. In any case, it can be verified that $\Nh(\chi_n) \le 4\sqrt{{\pi}/(3h)}$.
%
%\[
%  \Nh(\chi_n)\,\le\, 4\,\sqrt{\frac{\pi}/{3h}}\,.
%\]
%%
Thus
\begin{align*}
& {} \quad\sum_{n=n_h+1}^\infty \rez{n^{3/2} \ln n} \Nh(\chi_n) \le 4\,\sqrt{\frac{\pi}{3h}} \sum_{n=n_h+1} \rez{n^{3/2}\ln n}\\[2ex]
& {}\le {} 4\,\sqrt{\frac{\pi}{3h}} \,\rez{\ln n_h} \int_{n_h}^\infty \frac{dx}{x^{3/2}}%\\[2ex]
  \le 32 \,\sqrt{\frac{\pi}{3}} \,\rez{\ln n_h}\,\le\,\frac{32}{\ln 2}\,\sqrt{\frac{\pi}{3}}\,.
\end{align*}
Altogether we have shown that $\sup_{0<h\le 1} \Nh(f)< \infty$, which guarantees that $f\in \Mn.$

Next, assume towards a contradiction that $f\in M^{2,1}_\ast$. Then the series \eqref{scaled_mod1} would converge uniformly for $h\in (0, 1]$. In particular, there would exist an $n_0\in\N$ not depending on $h$ such that
\begin{equation}\label{uniform_convx}
  \sum_{n=n_0}^\infty\brbigg{\rez{h} \int_{I_{n,h}} \bigabs{\widehat{f}(v)}^2\dv}^{1/2}  < \sqrt{\frac{\pi}{3}}
\end{equation}
for all $h\in(0,1]$. We shall show that this is not possible.

Since
\[
   \sqrt{\frac{n+1}{h}} - \sqrt{\frac{n}{h}} >  \rez{2\sqrt{(n+1)h}}\,,
\]
we conclude in view of \zit{proofM-M1} that for the number $m_{n,h}$ of intervals $K_\mu$ which are subsets of $I_{n,h}$ we have
\bgl{proofM-M4}
   m_{n,h}\,\ge\, \rez{2\sqrt{(n+1)h}} - 1
\egl
provided that the right-hand side is greater than $1$. Now let $ 0 < h< \rez{16(n_0+1)}$
%$$ 0 \,<\, h\,<\, \rez{16(n_0+1)}$$
%
and define $N_h$ as the largest integer satisfying $N_h < \rez{16h} -1$.
%
%$$ N_h\,<\, \rez{16h} -1\,.$$
Then for $n\in [n_0, N_h]$, we have
\[
   \int_{I_{n,h}} \bigabs{\wh{f}(v)}^2 dv \ge  m_{n,h}\, \frac{8\pi}{3} \pBig{\frac{h}{n+1}}^{3/2} \rez{\ln^2\pbig{\sqrt{(n+1)/h})}}\,.
\]
Employing \zit{proofM-M4}, we find that
\begin{align*}
\sum_{n=n_0}^{N_h}\brbigg{\rez{h} \int_{I_{n,h}} \absbig{\wh{f}(v)}^2 dv}^{1/2}
                              & {}\ge {} \sqrt{\frac{8\pi}{3}}\, \sum_{n=n_0}^{N_h} \rez{(n+1) \ln\pbig{(n+1)/h)}} \\[2ex]
& {}\ge {} \sqrt{\frac{8\pi}{3}}\, \rez{\ln\pbig{(N_h+1)/h)}}\sum_{n=n_0}^{N_h} \rez{n+1}\,.
\end{align*}
The right-hand side approaches $\sqrt{2\pi/3}$ as $h\to 0+$. This contradicts \eqref{uniform_convx}.
\end{proof}

\section{A short biography of J.\,L.\,B.~Cooper}
Jacob Lionel Bakst Cooper, born on December 27, 1915 in Beaufort West, South Africa, entered the South African College School in Cape
Town in 1924. Upon his father's death his mother Franny (n\'{e}e Bakst) moved there with Lionel and his younger sister,
Gladys, to live with her parents, her father being a rabbinical scholar and her mother widely read. Perhaps his
maternal grandparents laid the basis to his unique personality. In this respect one of the authors (P.\,L.\,B.) in his
address given on the occasion of the funeral service of Lionel on August~14, 1979 said: ``Cooper had a sharp intellect,
always interested in the basic assumptions of the problems studied. He was a scholar in the old sense of the word,
widely read, having brilliant ideas, an inspiration to those who knew him. He did not seek the limelight, and was
somewhat reserved in public. He worked in a quiet way but still with great influence. He radiated authority in every
situation of life, an authority based on deep respect and justice. He had a healthy self-confidence which allowed him
to be composed; there was no rushing about him.'' (See \cite{Butzer_1981} and
\texttt{www-history.mcs.st-and.ac.uk/Biographies/Cooper.html}.).
Alan Hill in his tribute \cite{Hill_1981} writes that: ``\ldots\ when required he could be forcible---even fierce---in
his attitude \dots\ interested \dots\ in seeing that people were treated with decency and justice.''

%In fact, when he heard that a paper by two of the present authors had been rejected by a mathematical journal
%stationed at Oxford in 1975, he wrote to %the chief editor that at the present time Britain had no analyst who could assess
%the value of the paper correctly.

Cooper entered the University of South Africa in 1932. In view of his broad and great abilities he was encouraged to
become a rabbi but studied instead mathematics and physics, and received his B.\,Sc.~in 1935. He was also active in
student politics, and held strong views against racism and Nazism. While still at school he joined the Communist Party.
He told one of his daughters that this was because he felt the major injustice in South Africa was caused by the race
laws, and  only the communists were fighting for a fair society for all races. However he never forgot his Jewish
origins. Lionel won numerous prizes, including one for pure mathematics, one for applied mathematics and one for
history.

In 1935 he came to England as a Rhodes scholar to study at Queen's College, Oxford. As the Cooper family reported,
Lionel found the undergraduate syllabus at Oxford behind that in Cape Town, certainly in analysis which he found
trivial by comparison, and probably in outlook. The emphasis was on geometry, but classical, and not that developed in
the late 19th and early 20th century. Lionel obtained his D.~Phil.\ under Edward Titchmarsh's supervision in 1940 with
the thesis ``Theory and applications of Fourier integrals''. There he was very lucky to have met Kathleen Dixon, who
studied history at Oxford. They were married 1940, and their four children, Barbara (MSc.from Toronto), Frances (PhD.\
from Sussex), David (PhD.\ from Surrey), Deborah (PhD.\ from Swansea) all read mathematics. As to E.\,C.~Titchmarsh,
two of whose proofs play an important role in our derivative-free error estimates, as observed, P.\,L.\,B.\ had the
fortune to attend his invited lecture at the IMC in Amsterdam 1954. Lionel wrote the obituary address of his teacher
Titchmarsh \cite{Cooper_1963}.

During the early years of WW II he worked in the aircraft industry at Bristol before joining Birkbeck College, London,
in 1944 as Lecturer, becoming Reader in 1948. Kathleen recalls Lionel enjoying collaborating in his early years with
Hans Hamburger (1889--1956), who was Lecturer at Southampton's University College from 1941 to 1947 when he left for
Turkey, before returning to Cologne in 1953.

In 1951 Lionel Cooper was appointed Professor of Mathematics and Head of Department at University College, Cardiff,
Wales. There he stayed until 1963. At Cardiff he first had to put his whole energy into reorganizing and reorienting
the Department. The existing courses in Pure Mathematics were decidedly antiquated; applied mathematics had dominated
the scene. He quickly brought about, almost single-handedly, a revolution in pure mathematics, introducing a variety of
forward-looking courses of high quality, with functional analysis given prominence. He brought research to the
forefront by giving advanced courses and seminars, and adding new faculty members. The whole activity, which had
limited the time available for his own research, came to a break in 1954 when he spent a year at Witwatersrand
University. After spending the years 1964/65 at Caltech and 1965--67 as Full Professor at the University of Toronto, he
returned to England in 1967 to become Head of mathematics at the newly constituted Chelsea College of Science and
Technology of the University of London. Kathleen also recalls that Lionel initiated mathematical summer schools for
sixth form pupils from inner city schools in London in order to help them getting the grades for university as well as
a feel for going to university. In Cardiff he had already given courses for school teachers to activate them in the
dissemination of mathematics.
He died on August~8, 1979 in London after a heart operation.

Cooper's research was on a wide range of different topics in mathematical analysis: integral transform theory on the
real line and on groups, functional analysis, operator theory, essentially operators in Hilbert space, differential
equations, and thermodynamics. One of his chief tools was  Fourier transform theory. Details are to be found in David
Edmunds' obituary address \cite{Edmunds_1981} in which David treated Cooper's work in Functional Analysis and Differential
Equations, B.\,Sz.-Nagy that in Operator Theory, Butzer and R.\,J.~Nessel that in Transform Theory, and J.~Serrin that
in Thermodynamics. All in all Cooper wrote at least 50 mathematical papers in various journals throughout the world.

Of interest are two letters of Einstein dated 1949 addressed to Lionel which David forwarded to the authors and which are replies to two letters by Lionel of October and November of 1949. The first is a draft of his article \cite{Cooper_1950}. The matter concerns the famous Einstein-Podolky-Rosen (EPR) incompleteness argument in quantum theory, expounded by the three authors in their paper of 1935. The physicist Max Jammer (1915--2010) devotes some three pages of his excellent book \cite{Jammer_1974} to the arguments presented by Lionel.

His many research students included E.~Benham-Dehkordy (Iran), D.\,E. Davies (Farnborough), B.\,P.~Duggal (Nairobi), Robert Edmund Edwards (Canberra), C.\,E.~Finol (Vene\-zu\-ela), G.\,G.~Gould (Cardiff), Finbarr Holland (Cork), M.\,B.\ Sadiq (Iran), and James D.~Stewart (who sponsored the James Stewart Mathematics Centre at McMaster University). David Eric Edmunds (Brighton), who received his PhD.\ under Rosa Morris at Cardiff in 1955, reported that although he was not one of Lionel's research students, Lionel had by far the strongest influence on him, and it is due to him that he became an analyst, particularly in functional analysis and partial differential equations.

John Fournier (UBC, Vancouver) reported that Lionel, just before assuming his position at Toronto, taught that summer of 1964 a course on Fourier  Analysis in Madison, Wisconsin.  He as well as Charles Dunkl and Alan Schwartz attended that course. He further recalled that Walter Bloom, Garth Gaudry and John Price were students of R.\,E.~Edwards at Australian National University, thus academic grandsons of Lionel. Jim Stewart and John Fournier were classmates as undergraduates at the University of Toronto (1959--1963) and they wrote their amalgam paper \cite{Fournier-Stewart_1985} during Jim's sabbatical at UBC in 1983. Jim also reported that Lionel ``was a lovely, gentle man, full of good ideas, who died far too early'', and that after receiving his doctorate under Cooper at Toronto in 1967 he spent two full years with him as a postdoctorate fellow at Chelsea  College as well as a sabbatical there in 1977.

When at Aachen we planned to write the book on semi-group operators with Hubert Berens (1967) and on Fourier analysis and approximation with Rolf Nessel (1971) we knew we could always turn to Lionel for help when we were stuck. His solution to a specific problem always came quickly and he also  gave us the so important necessary confidence. Lionel came to the first Oberwolfach conference in approximation of 1963 (to which he brought with him Kathleen and his four children, the youngest being four at the time), as well as to the triennial conferences on approximation and operator theory conducted by P.L.B. and B.\,Sz.-Nagy at Oberwolfach from 1968 to 1984.

But mathematics was by no means Cooper's sole interest. While preparing a lecture tour for P.\,L.\,B. to Britain in 1973, he asked him to bring along a certain book of poems by R.\,M.~Rilke---he had become fluent in German through his friendship with German-Jewish refugees already in Cape Town. In fact, he also read or spoke French, Italian, Africaans and Russian. He played tennis and loved to walk, especially in the Lake District. He was a lover of music.

\section*{Acknowledgements}
The authors are  thankful to the referees for their constructive suggestions which were taken into account. They led to a more elegant presentation of the paper. In particular, the recommendation of using the theory of Besov spaces in Sections~\ref{sec_Lipschitz} and \ref{sec_Wiener_amalgam_spaces} resulted in considerably shorter proofs of Theorem~\ref{thm_equiv_new}, Corollary~\ref{cor_Lip_inclusions} and the Proposition~\ref{propa1}, which is essentially known.

%\section{References}
\bibliographystyle{elsarticle-num}
\bibliography{modulation}

\end{document}